\documentclass[11pt,twoside]{article}
\usepackage{mathrsfs}
\usepackage{amssymb}
\usepackage{amsmath}
\usepackage[all]{xy}
\usepackage{amsthm}

\usepackage{url}
\numberwithin{equation}{section}

\newtheorem{theorem}{Theorem}[section]
\newtheorem{proposition}[theorem]{Proposition}
\newtheorem{definition}[theorem]{Definition}
\newtheorem{remark}[theorem]{Remark}
\newtheorem{lemma}[theorem]{Lemma}
\newtheorem{example}[theorem]{Example}

\newtheorem{corollary}[theorem]{Corollary}
\newtheorem{question}[theorem]{Question}
\newtheorem{theorem*}{Theorem}

\newcommand{\Tor}{\operatorname{Tor}}
\newcommand{\Mod}{\operatorname{Mod}}
\newcommand{\Hom}{\operatorname{Hom}}
\newcommand{\Ext}{\operatorname{Ext}}
\newcommand{\E}{\mathbb{E}}
\newcommand{\F}{\mathbb{F}}

\newcommand{\X}{\mathcal{X}}

\newcommand{\I}{\mathcal{I}}
\newcommand{\J}{\mathcal{J}}
\newcommand{\id}{\operatorname{id}}
\newcommand{\Ob}{\operatorname{Ob}}
\newcommand{\Mor}{\operatorname{Mor}}
\newcommand{\ra}{\rightarrow}

\newcommand{\s}{\mathfrak{s}}

\newcommand{\C}{\mathscr{C}}
\newcommand{\Ph}{\bf{Ph}}
\newcommand{\Coph}{\bf{Coph}}
\newcommand{\inj}{\bf{inj}}
\newcommand{\proj}{\bf{proj}}

\title{ \bf Phantom Ideals and Cotorsion Pairs in Extriangulated Categories
\thanks{2010 Mathematics Subject Classification: 18G25, 16E30, 18E40.}
\thanks{Keywords: Phantom ideals; Cotorsion pairs; Extriangulated categories;
(Co)phantom morphisms; Special precovering ideals; Special preenveloping ideals.}}
\vspace{0.2cm}

\author{Tiwei Zhao\thanks{E-mail address:  tiweizhao@hotmail.com}, \
Zhaoyong Huang\thanks{E-mail address: huangzy@nju.edu.cn} \\
{\it \footnotesize  Department of Mathematics, Nanjing University, Nanjing 210093, Jiangsu Province, P.R. China}}
\date{ }
\begin{document}

\baselineskip=16pt
\maketitle

\begin{abstract}
In this paper, we introduce and study relative phantom morphisms in extriangulated categories defined
by Nakaoka and Palu. Then using their properties, we show that if $(\C,\E,\s)$ is an extriangulated category with enough injective objects
and projective objects, then there exists a bijective correspondence between any two of the following classes:
(1) special precovering ideals of $\C$; (2) special preenveloping ideals of $\C$; (3) additive subfunctors of $\E$ having enough
special injective morphisms; and (4) additive subfunctors of $\E$ having enough special projective morphisms.
Moreover, we show that if $(\C,\E,\s)$ is an extriangulated category with enough injective objects
and projective morphisms, then there exists a bijective correspondence between the following two classes:
(1) all object-special precovering ideals of $\C$; (2) all additive subfunctors of $\E$ having enough special injective objects.
\end{abstract}

\pagestyle{myheadings}
\markboth{\rightline {\scriptsize   T. Zhao, Z. Huang}}
         {\leftline{\scriptsize  Phantom Ideals and Cotorsion Pairs in Extriangulated Categories}}


\section{Introduction} 

In algebra, geometry and topology, exact categories and triangulated categories are two fundamental structures.
The interest of exact categories is manifold,  and there is no need to argue that they are both useful and
important. Triangulated categories were introduced in the mid 1960¡¯s by Verdier \cite{Ve}
in his thesis. Having their origins in algebraic geometry and
algebraic topology, triangulated categories have also become indispensable
in many different areas of mathematics by now. As expected, exact categories and triangulated categories are not
independent of each other. A well-known fact is that triangulated categories which at the same time are
abelian must be semisimple. Also, there are a series of ways to produce triangulated
categories from abelian ones, such as, taking the stable categories of Frobenius exact categories, or taking the
homotopy categories or derived categories of complexes over abelian categories. On the other hand, because of the recent
development of the cluster theory, it becomes possible to produce abelian categories from triangulated ones, that is,
starting from a cluster category and taking a cluster tilting subcategory, one can get a suitable quotient category,
which turns out to be abelian \cite{KZ}. In addition, exact categories and triangulated categories possess same properties
in many homological invariants, for example in the aspect of  approximation theory \cite{AbNa,Liu,Na}.
Approximation theory is the main part of relative homological algebra and representation theory of algebras,
and its starting point is to approximate arbitrary objects by a class of
suitable subcategories. In this process, the notion of cotorsion pairs provides a fruitful context,
in particular, it is closely related to many important homological structures, for example, $t$-structure, co-$t$-structure,
cluster tilting subcategories, and so on. In general, to transfer the homological properties between exact categories and
triangulated categories, one needs  to specify to the case of stable categories of Frobenius exact categories, and then
lift (or descend) the associated definitions and statements, and finally adapt the proof so that it can apply to any exact
(or triangulated) categories. However, it is not easy to do it in general case, especially in the third step.
To overcome the difficulty, Nakaoka and Palu \cite{NP} introduced the notion of externally triangulated categories
(extriangulated categories for short) by a careful looking what is necessary in the definition of cotorsion pairs in
exact and triangulated cases. Under this notion, exact categories and extension-closed subcategories of triangulated
categories both are externally triangulated, and hence, in some levels,  it becomes easy to give  uniform statements
and proofs for the exact and triangulated settings \cite{NP,ZZ}.



In an abstract category, objects and morphisms  are two essential components;
and by a well-known embedding from a category to its morphism morphism, objects can be viewed as special morphisms.
In the classical approximation theory, we mainly concern the objects and the associated subcategories. However, in general case,
it seems that the morphisms and the associated ideals also should be concerned in the approximation theory.
From this point of view, Fu, Guil Asensio, Herzog and Torrecillas in \cite{FGHT} introduced the notion of
ideal cotorsion pairs and developed  the ideal approximation theory of exact categories.
Inside it, the phantom ideal plays an important role in the aspect
of providing a certain ideal cotorsion pair; and it has been investigated in
algebraic topology \cite{Mc}, stable homotopy categories of spectra \cite{AdW},
triangulated categories \cite{BM,Ne}, and stable categories of finite group rings
\cite{Ben,BenG99,BenG01}. In particular, Herzog generalized in \cite{He07}
the phantom morphism to the category of left $R$-modules of arbitrary associative
ring $R$ in the following way:
a morphism $f:M\ra N$ of left $R$-modules is called a \emph{phantom morphism} if the natural
transformation $\Tor_1^R(-,f):\Tor^R_1(-,M)\ra \Tor^R_1(-,N)$ is zero, or equivalently,
the pullback of any short exact sequence along $f$ is pure exact. Then he showed
that every module admits a phantom cover. As a generalization of the (classical) approximation theory
for subcategories, Fu et. al developed in \cite{FGHT} the approximation
theory of an exact category $\mathscr{A}$ for ideal cotorsion pairs.  A careful look reveals that
the essentially necessary matters in \cite{FGHT} are pullbacks and pushouts, that is, some special
operations of functors. So this inspires us to establish the approximation theory in
an additive category equipped with an additive bifunctor; in particular,
we consider it in extriangulated categories, which not only unifies the ideal approximation theory in
exact categories and triangulated categories, but also extends this theory to those categories
which are neither exact nor triangulated as much as possible.

This paper is organized as follows.

In Section 2, we give some terminology and some preliminary results.

In Section 3, we first introduce the notion of relative phantom morphisms in an additive category,
and then extend it to an extriangulated category. We study the relationship between relative phantom
morphisms and relative injective morphisms, and give a sufficient condition such that they form
a relative cotorsion pair.

In Section 4, we mainly discuss the role of phantom operations, and use it to  investigate the interplay
among special precovering ideals, special preenveloping ideals,  additive subfunctors having enough
special injective morphisms, and additive subfunctors having enough special projective morphisms.
We show that if $(\C,\E,\s)$ is an extriangulated category with enough injective objects and projective objects,
then we have the following bijective correspondences.
$$ \ \
\xymatrix@=1.5cm{
{\begin{tabular}{|c|}
  \hline
      \mbox{all special precovering ideals} \\ \mbox{of $\C$}\\
  \hline
\end{tabular}}\ar@<+4pt>[r]^{\!\!\!\!\!\!\!\!\!\!\!\!\!\!\!(-)^\star}\ar@<+4pt>[d]^{(-)^{\perp_\E}}&
\ar@<+4pt>[d]^G\ar@<+4pt>[l]^{\!\!\!\!\!\!\!\!\!\!\!\!\!\!\!\!\Ph(-)}\ar[ld]_{\!\!\!\!\!\!\!\!\!\!\!\!\!\!\!\!\!\!\!\!\!\!\!\!\!\!\!\!\!\!\!\!\!\!\!\!\!
\!\!\!\!\!\!\!\!\!\!\!\!\!\!\!\!\!\!\!\!\!\!\!\!\!\!\!\!\!\!\!\!\!\!\!\!\!\!\!\!\!\!\!(-)\mbox{-}\proj}{\begin{tabular}{|c|}
  \hline
        \mbox{all additive subfunctors of $\E$ having}\\\mbox{enough special injective morphisms} \\
  \hline
\end{tabular}}\\
{\begin{tabular}{|c|}
  \hline
      \mbox{all special preenveloping ideals} \\ \mbox{of $\C$}\\
  \hline
\end{tabular}}\ar@<+4pt>[r]^{\!\!\!\!\!\!\!\!\!\!\!\!\!\!\!(-)_\star}\ar@<+4pt>[u]^{^{\perp_\E}(-)}
&\ar@<+4pt>[u]^F\ar@<+4pt>[l]^{\!\!\!\!\!\!\!\!\!\!\!\!\!\!\!\!\Coph(-)}
\ar[lu]_{~~~~~~~~~~~~~~~~(-)\mbox{-}\inj}{\begin{tabular}{|c|}
  \hline
        \mbox{all additive subfunctors of $\E$ having}\\\mbox{enough special projective morphisms} \\
  \hline
\end{tabular}}}
$$
Here $F={(-)^\star}\circ{^{\perp_\E}(-)}\circ\Coph(-)$ and $G={(-)_\star}\circ{{(-)^{\perp_\E}}}\circ{\Ph(-)}$,
see Section 3 for the definitions of these functors.

In Section 5, we consider object-special precovering ideals, and show that  if $(\C,\E,\s)$ is an extriangulated category
with enough injective objects and projective morphisms, then we have the following bijective correspondence.
$$
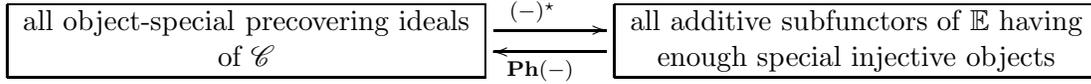
\xymatrix@C=1.5cm{
{\begin{tabular}{|c|}
  \hline
      \mbox{all object-special precovering ideals} \\ \mbox{of $\C$}\\
  \hline
\end{tabular}}\ar@<+4pt>[r]^{\!\!\!\!\!\!\!\!\!\!(-)^\star}&\ar@<+4pt>[l]^{\!\!\!\!\!\!\Ph(-)}{\begin{tabular}{|c|}
  \hline
        \mbox{all additive subfunctors of $\E$ having}\\\mbox{enough special injective objects} \\
  \hline
\end{tabular}}}
$$

\section{Preliminaries}
Throughout this paper, $\C$ is an additive category and $\E:\C^{\operatorname{op}}\times\C\rightarrow \mathfrak{A}b$
is a biadditive functor, where $\mathfrak{A}b$ is the category of abelian groups.
\subsection{$\E$-extensions}
\begin{definition}{\rm
(\cite[Definition 2.1,2.5]{NP}) For any $A,C\in\C$, there is a corresponding abelian group $\E(C,A)$.
\begin{enumerate}
\item[(1)] An element $\delta\in\E(C,A)$ is called an \emph{$\E$-extension}. More formally, an $\E$-extension is a triple $(A,\delta,C)$.
\item[(2)]The zero element $0$ in $\E(C,A)$ is called the \emph{split $\E$-extension}.
\end{enumerate}}
\end{definition}

Let $a\in \C(A,A')$ and $c\in\C(C',C)$. Then we have the following commutative diagram
$$\xymatrix@C=0.5cm{
  \E(C,A) \ar[rr]^{\E(C,a)}\ar[d]^{\E(c,A)}\ar[rrd]^{\E(c,a)} && \E(C,A')\ar[d]^{\E(c,A')} \\
 \E(C',A)\ar[rr]^{\E(C',a)} &&\E(C',A') }
$$  in $\mathfrak{A}b$.
For an $\E$-extension $(A,\delta,C)$, we briefly write $a_\star\delta:=\E(C,a)(\delta)$ and
$c^\star\delta:=\E(c,A)(\delta)$. Then
$$\E(c,a)(\delta)=c^\star a_\star \delta= a_\star c^\star \delta.$$

\begin{definition}
{\rm (\cite[Definition 2.3]{NP})
Given two $\E$-extensions $(A,\delta,C)$ and $(A',\delta',C')$. A \emph{morphism} from $\delta$ to $\delta'$
is a pair $(a,c)$ of morphisms, where $a\in \C(A,A')$ and $c\in \C(C,C')$, such that $a_\star\delta=c^\star\delta$.
In this case, we denote it by $(a,c):\delta\rightarrow \delta'.$}
\end{definition}

Now let $A,C\in\C$. Two sequences of morphisms
$$
\xymatrix@C=0.5cm{
  A \ar[r]^x & B \ar[r]^y & C }
  \mbox{ and }
  \xymatrix@C=0.5cm{
  A \ar[r]^{x'} & B' \ar[r]^{y'} & C}
$$
are said to be \emph{equivalent} if there exists an isomorphism $b\in\C(B,B')$ such that the following diagram
$$
\xymatrix{
   A \ar[r]^x\ar@{=}[d] & B \ar[r]^y\ar[d]^b_\cong & C\ar@{=}[d]\\
    A \ar[r]^{x'} & B' \ar[r]^{y'} & C }
$$
commutes.
We denote by $[\xymatrix@C=0.5cm{
A \ar[r]^x & B \ar[r]^y & C }]$ the equivalence class of $\xymatrix@C=0.5cm{
A \ar[r]^x & B \ar[r]^y & C}$. In particular, we write $0:=[\xymatrix@C=0.5cm{
A \ar[r]^{\!\!\!\!\!\!\!{1\choose 0}} & A\oplus C \ar[r]^{~~~(0 \ 1)} & C }]$.

Note that, for any pair $\delta\in \E(C,A)$ and $\delta'\in \E(C',A')$, since $\E$ is biadditive,
there exists a natural isomorphism
$$\E(C\oplus C',A\oplus A')\cong \E(C,A)\oplus\E(C,A')\oplus\E(C',A)\oplus\E(C',A').$$
We define the symbol $\delta \oplus\delta'$ to be the element in $\E(C\oplus C',A\oplus A')$ corresponding to the element
$(\delta,0,0,\delta')$ in $\E(C,A)\oplus\E(C,A')\oplus\E(C',A)\oplus\E(C',A')$ through the above isomorphism.

\begin{definition}
{\rm (\cite[Definition 2.9]{NP})
Let $\s$ be a correspondence which associates an equivalence class $\s(\delta)=[\xymatrix@C=0.5cm{
A \ar[r]^x & B \ar[r]^y & C }]$ to each $\E$-extension $\delta\in\E(C,A)$. The $\s$ is called a
\emph{realization} of $\E$ provided that it satisfies the following condition.
\begin{enumerate}
\item[(R)]  Let $\delta\in\E(C,A)$ and $\delta'\in\E(C',A')$ be any pair of $\E$-extensions with
$$\s(\delta)=[\xymatrix@C=0.5cm{
A \ar[r]^x & B \ar[r]^y & C }]
\mbox{ and }
\s(\delta')=[ \xymatrix@C=0.5cm{
A' \ar[r]^{x'} & B' \ar[r]^{y'} & C' }].$$
Then for any morphism $(a,c):\delta\rightarrow\delta'$, there exists $b\in\C(B,B')$ such that the following diagram
$$\xymatrix{
A \ar[r]^x\ar[d]^a & B \ar[r]^y\ar[d]^b & C\ar[d]^c\\
A' \ar[r]^{x'} & B' \ar[r]^{y'} & C' }$$
commutes.
\end{enumerate}}
\end{definition}

Let $\s$ be a realization of $\E$. If $\s(\delta)=[\xymatrix@C=0.5cm{
A \ar[r]^x & B \ar[r]^y & C }]$ for some $\E$-extension $\delta\in\E(C,A)$, then we say that the sequence
$\xymatrix@C=0.5cm{A \ar[r]^x & B \ar[r]^y & C }$ \emph{realizes} $\delta$; and in the condition (R),
we say that the triple $(a,b,c)$ \emph{realizes} the morphism $(a,c)$.

\begin{remark}{\rm
Let $\s$ be a realization of $\E$, and let $\delta\in\E(C,A)$ be an $\E$-extension with
$\s(\delta)=[\xymatrix@C=0.5cm{A \ar[r]^x & B \ar[r]^y & C }]$.
\begin{enumerate}
\item[(1)] For any $a\in\C(A,A')$, since $a_\star\delta={\id_C}^\star a_\star\delta$, there exists a morphism
$(a,\id_C):\delta\rightarrow a_\star \delta$. Assume that
$$\s(a_\star \delta)=[ \xymatrix@C=0.5cm{
A' \ar[r]^{x'} & B' \ar[r]^{y'} & C }].$$
Then by the condition (R), there exists a commutative diagram
$$\xymatrix{
A \ar[r]^x\ar[d]^a & B \ar[r]^y\ar[d] & C\ar@{=}[d]\\
A' \ar[r]^{x'} & B' \ar[r]^{y'} & C.}$$
\item[(2)]  For each $c\in\C(C',C)$, since ${\id_A}_\star c^\star\delta=c^\star\delta$, there exists a morphism
$(\id_A,c):c^\star\delta\rightarrow \delta$. Assume that
$$\s(c^\star\delta)=[ \xymatrix@C=0.5cm{
A \ar[r]^{x''} & B'' \ar[r]^{y''} & C' }].$$
Then by the condition (R), there exists a commutative diagram
$$\xymatrix{
A \ar[r]^{x''}\ar@{=}[d] & B'' \ar[r]^{y''}\ar[d] & C'\ar[d]^c\\
A \ar[r]^{x} & B \ar[r]^{y} & C.}$$
\end{enumerate}}
\end{remark}

For any two equivalence classes $[\xymatrix@C=0.5cm{
  A \ar[r]^x & B \ar[r]^y & C }]$ and $[\xymatrix@C=0.5cm{
  A' \ar[r]^{x'} & B' \ar[r]^{y'} & C' }]$, we define
$$[\xymatrix@C=0.5cm{
  A \ar[r]^x & B \ar[r]^y & C }]\oplus[\xymatrix@C=0.5cm{
  A' \ar[r]^{x'} & B' \ar[r]^{y'} & C' }]:=[\xymatrix@C=0.5cm{
  A\oplus A' \ar[r]^{x\oplus x'} & B\oplus B' \ar[r]^{y\oplus y'} & C\oplus C' }].$$

\begin{definition}
{\rm (\cite[Definition 2.10]{NP})
A realization $\s$ of $\E$ is called \emph{additive} if it satisfies the following conditions.
\begin{enumerate}
\item[(1)] For any $A,C\in\C$, the split $\E$-extension $0\in\E(C,A)$ satisfies $\s(0)=0$.
\item[(2)] For any pair of $\E$-extensions $\delta\in\E(C,A)$ and $\delta'\in\E(C',A')$,
we have $\s(\delta\oplus\delta')=\s(\delta)\oplus\s(\delta')$.
\end{enumerate}}
\end{definition}

Let $\s$ be an additive realization of $\E$. By \cite[Remark 2.11]{NP}, we have that if the sequence
$\xymatrix@C=0.5cm{A \ar[r]^x & B \ar[r]^y & C }$ realizes $0$ in $\E(C,A)$, then $x$ is a section and $y$ is a retraction.

\subsection{Externally triangulated categories}

\begin{definition}{\rm
(\cite[Definition 2.12]{NP}) Let $\C$ be an additive category. We call the triple $(\C,\E,\s)$ an \emph{externally triangulated category}
(or \emph{extriangulated category} for short) if it satisfies the following conditions.
 \begin{enumerate}
\item[(ET1)] $\E:\C^{\operatorname{op}}\times\C\rightarrow \mathfrak{A}b$
is a biadditive functor.
\item[(ET2)] $\s$ is an additive realization of $\E$.
\item[(ET3)] Let $\delta\in\E(C,A)$ and $\delta'\in\E(C',A')$ be any pair of $\E$-extensions with
$$
\s(\delta)=[\xymatrix@C=0.5cm{
  A \ar[r]^x & B \ar[r]^y & C }]
  \mbox{ and }
 \s(\delta')=[ \xymatrix@C=0.5cm{
  A' \ar[r]^{x'} & B' \ar[r]^{y'} & C' }].
$$
For any commutative diagram
$$
\xymatrix{
   A \ar[r]^x\ar[d]^a & B \ar[r]^y\ar[d]^b & C\\
    A' \ar[r]^{x'} & B' \ar[r]^{y'} & C' }
$$
in $\C$, there exists a morphism $(a,c):\delta\ra \delta'$ which is realized by the triple $(a,b,c)$.
\item[${\rm (ET3)^{op}}$]  Let $\delta\in\E(C,A)$ and $\delta'\in\E(C',A')$ be any pair of $\E$-extensions with
$$
\s(\delta)=[\xymatrix@C=0.5cm{
  A \ar[r]^x & B \ar[r]^y & C }]
  \mbox{ and }
 \s(\delta')=[ \xymatrix@C=0.5cm{
  A' \ar[r]^{x'} & B' \ar[r]^{y'} & C' }].
$$
For any commutative diagram
$$
\xymatrix{
   A \ar[r]^x & B \ar[r]^y\ar[d]^b & C\ar[d]^c\\
    A' \ar[r]^{x'} & B' \ar[r]^{y'} & C' }
$$
in $\C$, there exists a morphism $(a,c):\delta\ra \delta'$ which is realized by the triple $(a,b,c)$.
\item[(ET4)]   Let $\delta\in\E(C,A)$ and $\rho\in\E(F,B)$ be any pair of $\E$-extensions with
$$
\s(\delta)=[\xymatrix@C=0.5cm{
  A \ar[r]^x & B \ar[r]^y & C }]
  \mbox{ and }
 \s(\rho)=[ \xymatrix@C=0.5cm{
  B \ar[r]^{u} & D \ar[r]^{v} & F }].
$$
Then there exist an object $E\in\C$, an $\E$-extension $\xi$ with $\s(\xi)=[\xymatrix@C=0.5cm{
  A \ar[r]^z & D \ar[r]^w & E }]$, and a commutative diagram
$$
\xymatrix{
   A \ar[r]^x\ar@{=}[d] & B \ar[r]^y\ar[d]^u & C\ar[d]^s\\
    A \ar[r]^{z} & D \ar[r]^{w}\ar[d]^v &E\ar[d]^t\\
    &F\ar@{=}[r] &F }
$$
in $\C$, which satisfy the following compatibilities.
 \begin{enumerate}
\item[(i)] $\s(y_\star\rho)=[\xymatrix@C=0.5cm{
  C \ar[r]^s & E \ar[r]^t & F }].$
\item[(ii)] $s^\star\xi=\delta$.
\item[(iii)] $x_\star\xi=t^\star\rho$.
\end{enumerate}
\item[${\rm (ET4)^{op}}$]   Let $\eta\in\E(E,A)$ and $\xi\in\E(F,C)$ be any pair of $\E$-extensions with
$$
\s(\eta)=[\xymatrix@C=0.5cm{
  A \ar[r]^z & D \ar[r]^w & E }]
  \mbox{ and }
 \s(\xi)=[ \xymatrix@C=0.5cm{
  C \ar[r]^{s} & E \ar[r]^{t} & F }].
$$
Then there exist an object $B\in\C$, an $\E$-extension $\theta$ with $\s(\theta)=[\xymatrix@C=0.5cm{
  B \ar[r]^u & D \ar[r]^v & F }]$, and a commutative diagram
$$
\xymatrix{
   A \ar[r]^x\ar@{=}[d] & B \ar[r]^y\ar[d]^u & C\ar[d]^s\\
    A \ar[r]^{z} & D \ar[r]^{w}\ar[d]^v &E\ar[d]^t\\
    &F\ar@{=}[r] &F }
$$
in $\C$ satisfying the following compatibilities.
 \begin{enumerate}
\item[(i)] $\s(s^\star\eta)=[\xymatrix@C=0.5cm{
  A \ar[r]^x & B \ar[r]^y & C }].$
\item[(ii)] $y_\star\theta=\xi$.
\item[(iii)] $x_\star\eta=t^\star\theta$.
\end{enumerate}
\end{enumerate}}
\end{definition}

\begin{definition}{\rm
(\cite[Definition 2.19]{NP}) Let $(\C,\E,\s)$ be a triple satisfying (ET1) and (ET2).
 \begin{enumerate}
\item[(1)] If a sequence $\xymatrix@C=0.5cm{
  A \ar[r]^x & B \ar[r]^y & C }$ realizes an $\E$-extension $\delta\in\E(C,A)$, then we call the pair $(\xymatrix@C=0.5cm{
  A \ar[r]^x & B \ar[r]^y & C },\delta)$ an \emph{$\E$-triangle}, and write it in the following way
  $$
  \xymatrix@C=0.5cm{
  A \ar[r]^x & B \ar[r]^y & C\ar@{-->}[r]^\delta&. }
  $$
  In this case, $x$ is called an \emph{$\E$-inflation}, and $y$ is called an \emph{$\E$-deflation}.
\item[(2)] Let $
  \xymatrix@C=0.5cm{
  A \ar[r]^x & B \ar[r]^y & C\ar@{-->}[r]^\delta& }
  $ and $
  \xymatrix@C=0.5cm{
  A' \ar[r]^{x'} & B' \ar[r]^{y'} & C'\ar@{-->}[r]^{\delta'}& }
$ be any pair of $\E$-triangles. If a triple $(a,b,c)$ realizes $(a,c):\delta\ra\delta'$ as in the condition (R),
then we write it as
  $$
\xymatrix{
   A \ar[r]^x\ar[d]^a & B \ar[r]^y\ar[d]^b & C\ar[d]^c\ar@{-->}[r]^\delta& \\
    A' \ar[r]^{x'} & B' \ar[r]^{y'} & C' \ar@{-->}[r]^{\delta'}&, }
$$
and call the triple $(a,b,c)$ a \emph{morphism of  $\E$-triangles}.
\end{enumerate}
}
\end{definition}

We collect some examples of extriangulated categories as follows.

\begin{example}\label{2.8}
\begin{enumerate}
\item[]
\end{enumerate}
{\rm (1) All abelian categories are extriangulated categories. In fact, let $\mathcal{A}$ be an abelian category. Then
$\E:=\Ext^1_{\mathcal{A}}(-,-):\mathcal{A}^{\operatorname{op}} \times\mathcal{A}\rightarrow \mathfrak{Ab}$ and the
realization $\mathfrak{s}$ is defined by associating equivalence classes of short exact sequences to itself.

(2) Every subbifunctor $\mathcal{F}(-,-)\subseteq\Ext^1_{\mathcal{A}}(-,-)$ over an abelian category $\mathcal{A}$
induces an extriangulated category, where $\E:=\mathcal{F}(-,-)$ and its corresponding realization
$\mathfrak{s}:=\mathfrak{s}\mid_{\mathcal{F}}$. A trivial example is $\mathcal{F}=0$, that is, consider all split
short exact sequences over $\mathcal{A}$. Moreover, for example, let $R$ be a ring, recall that a left $R$-module $M$ is
called \emph{Gorenstein projective} if there exists an exact sequence
$$\xymatrix@C=0.5cm{\cdots\ar[r] & P_1\ar[r] &P_0\ar[r] &P^0\ar[r] &P^1\ar[r] &\cdots}$$
in $R$-$\Mod$ (the category of left $R$-modules) with all $P_i,P^i$ projective, such that it stays exact after applying
the functor $\Hom_R(-,P)$ for any projective left $R$-module $P$, and $M=\mbox{Im}(P_0\ra P^0)$. Dually, we the notion
of Gorenstein injective left $R$-modules is defined.
If moreover $R$ is a Gorenstein ring, that is, $R$ is a left and right Noetherian ring with finite left and right
self-injective dimensions, then we may get the corresponding Gorenstein derived functor $\mbox{GExt}^1_R(-,-)$ (\cite{EJ}).
In this case, we have an extriangulated category $(R\mbox{-Mod},\mbox{GExt},\mathfrak{s})$, where the $\mbox{GExt}$-triangles
are those short exact sequences in $R\mbox{-Mod}$ which stay exact after applying
the functor $\Hom_R(-,G)$ for any $G\in\mathcal{GP}$ (or equivalently, after applying
the functor $\Hom_R(H,-)$ for any $H\in\mathcal{GI}$). Here $\mathcal{GP}$ and $\mathcal{GI}$ stand for the full subcategories
of $R$-$\Mod$ consisting of all Gorenstein projective and injective left $R$-modules respectively.

Recall that a short exact sequence $0\ra A\ra B\ra C\ra 0$ in $R$-$\Mod$ is called \emph{pure exact}
if for any finitely presented left $R$-modules $F$, the induce sequence $\Hom_R(F,B)\ra \Hom_R(F,C)\ra 0$ is exact.
The pure injective (resp. pure projective) left $R$-modules are those modules which are injective (resp. projective)
with respect to all short pure exact sequences in $R\mbox{-Mod}$. It is well known that there exist enough pure injective
and pure projective objects in $R\mbox{-Mod}$. Following the corresponding pure projective and pure injective resolutions,
we have the cohomological functor $\mbox{PExt}^1_R(-,-)$ (\cite{KS,S}). Then $(R\mbox{-Mod},\mbox{PExt},\mathfrak{s})$ is an
extriangulated category, where the $\mbox{Pext}$-triangles are those short exact sequences in $R\mbox{-Mod}$ which are pure exact.

(3) Exact categories $\mathcal{C}$ such that $\Ext^1(-,-):\mathcal{C}^{op}\times\mathcal{C}\to \mathfrak{A}b$ is
a biadditive functor (especially, exact and small categories) are extriangulated categories, see \cite[Example 2.13]{NP}.
Note that for a ring $R$, the subcategory $\mathcal{GP}$ of $R\mbox{-Mod}$ is closed under extensions and hence it is in fact
an exact category. Thus we also have an extriangulated category $(\mathcal{GP},\mathcal{E},\mathfrak{s})$, where $\mathcal{E}$
is the collection of all short exact sequences  in $R\mbox{-Mod}$ whose terms in $\mathcal{GP}$.

(4) Triangulated categories are extriangulated categories. In details, let $\mathcal{T}$ be a triangulated category
and $[1]$ the shift functor. Set $\E:=\mathcal{T}(-,-[1])$, and for any $\delta\in \E(Z,X)=\mathcal{T}(Z,X[1])$,
choose a triangle $\xymatrix@C=0.5cm{X\ar[r]& Y\ar[r]& Z\ar[r]^{\delta} & X[1]}$ and define
$\mathfrak{s}(\delta):=[\xymatrix@C=0.5cm{X\ar[r]& Y\ar[r]& Z}]$, see \cite[Section 3.3]{NP}.

(5) All extension-closed subcategories of extriangulated categories are again extriangulated, see \cite[Remark 2.18]{NP}.

(6) Nakaoka and Palu in \cite{NP} provided a construction for which extriangulated categories are neither exact nor triangulated.
That is,  let  $\mathcal{T}$ be an extriangulated category and $\mathcal{X}$ a full subcategory of $\mathcal{T}$. Denote by
$\mathcal{P}$ (resp. $\mathcal{I}$) the full subcategory consisting of projective (resp. injective) objects in $\mathcal{T}$.
If $\mathcal{X}\subseteq \mathcal{P}\cup\mathcal{I}$, then the quotient category $\mathcal{T}/\mathcal{X}$ is an extriangulated category,
see \cite[Proposition 3.30]{NP} for more details.

(7) Zhou and Zhu in \cite[Corollary 4.10 and Remark 4.11]{ZZ} also provided a construction for which extriangulated categories
are neither exact nor triangulated. That is, let $\mathcal{T}$ be an extriangulated category with Auslander-Reiten translation $\tau$
and $\mathcal{X}$ a functorrially finite subcategory of $\mathcal{T}$ which satisfies $\tau\mathcal{X}=\mathcal{X}$. For any $X,Z\in \mathcal{T}$,
define $\E(Z,X)\subseteq\mathcal{T}(Z,X[1])$ to be the collection of equivalence classes of triangles
$\xymatrix@C=0.5cm{X\ar[r]^f& Y\ar[r]& Z\ar[r]^{\delta} & X[1]}$ such that $\mathcal{T}(f,X')$ is epic for any $X'\in\mathcal{X}$, and define
$\mathfrak{s}{\delta}:=[\xymatrix@C=0.5cm{X\ar[r]& Y\ar[r]& Z}]$. Then $(\mathcal{T},\E,\mathfrak{s})$ is a Frobenius extriangulated category.
If $\mathcal{X}\neq \{0\}$, then $(\mathcal{T},\E,\mathfrak{s})$ is not triangulated; and if $\mathcal{X}\neq \mathcal{T}$, then it is not exact.}
\end{example}

\begin{remark}{\rm
Let $(\C,\E,\s)$ be a triple satisfying (ET1) and (ET2), and let
$\xymatrix@C=0.5cm{
  A \ar[r]^x & B \ar[r]^y & C\ar@{-->}[r]^\delta& }
  $ be an $\E$-triangle.
\begin{enumerate}
\item[(1)] For any $a\in\C(A,A')$,  there exists a morphism of  $\E$-triangles
$$
\xymatrix{
   A \ar[r]^x\ar[d]^a & B \ar[r]^y\ar[d] & C\ar@{=}[d]\ar@{-->}[r]^\delta& \\
    A' \ar[r]^{x'} & B' \ar[r]^{y'} & C \ar@{-->}[r]^{a_\star\delta}&. }
$$
\item[(2)]  For any $c\in\C(C',C)$,  there exists a morphism of  $\E$-triangles
$$
\xymatrix{
   A \ar[r]^{x'}\ar@{=}[d] & B' \ar[r]^{y'}\ar[d] & C'\ar[d]^c\ar@{-->}[r]^{c^\star\delta}& \\
    A \ar[r]^{x} & B \ar[r]^{y} & C \ar@{-->}[r]^{\delta}&.}
$$
\end{enumerate}}
\end{remark}

We introduce the following

\begin{definition}
{\rm Let $(\C,\E,\s)$ be a triple satisfying (ET1) and (ET2). An object $E\in\C$ is said to be  \emph{injective} if for any $\E$-triangle
$\xymatrix@C=0.5cm{
  A \ar[r]^x & B \ar[r]^y & C\ar@{-->}[r]^\delta& }
  $ and each morphism $e\in\C(A,E)$, there exists $b\in\C(B,E)$ such that $e=bx$.}
\end{definition}

\begin{lemma}\label{2.11}
Let $(\C,\E,\s)$ be a triple satisfying (ET1), (ET2) and (ET3). Then the following statements are equivalent for an object $E\in\C$.
\begin{enumerate}
\item[(1)] $E$ is injective.
\item[(2)] $\E(C,E)=0$ for any $C\in\C$.
\item[(3)] Any $\E$-triangle
$\xymatrix@C=0.5cm{
E \ar[r] & B \ar[r] & C\ar@{-->}[r]^\delta& }$ splits.
  \end{enumerate}
\end{lemma}

\begin{proof}
(1) $\Rightarrow$ (2) Let $\delta\in\E(C,E)$ and $\s(\delta)=[\xymatrix@C=0.5cm{
E \ar[r]^e & B \ar[r]^y & C }]$. Since $E$ is injective by (1), there exists $b\in\C(B,E)$ such that $be=\id_E$;
that is, we have the following commutative diagram
$$\xymatrix{
E \ar[r]^e\ar[d]_{\id_E} & B \ar[r]^y\ar[d]^b & C\ar@{-->}[r]^\delta& \\
E \ar[r]^{{\id_E}} & E \ar[r] & 0 \ar@{-->}[r]^0& .}$$
By (ET3), we get a morphism of $\E$-triangles
$$\xymatrix{
E \ar[r]^e\ar[d]_{\id_E} & B \ar[r]^y\ar[d]^b & C\ar[d]^0\ar@{-->}[r]^\delta& \\
E \ar[r]^{{\id_E}} & E \ar[r] & 0 \ar@{-->}[r]^0& .}
$$
Thus we have $\delta=0^\star0=0$.

(2) $\Rightarrow$ (3) It is trivial.

(3) $\Rightarrow$ (1)  Let $\xymatrix@C=0.5cm{
A \ar[r]^x & B \ar[r]^y & C\ar@{-->}[r]^\delta& }$ be any $\E$-triangle.
Then for any $a\in\C(A,E)$, there exists a morphism of $\E$-triangles
$$
\xymatrix{
   A \ar[r]^x\ar[d]^a & B \ar[r]^y\ar[d]^b & C\ar@{=}[d]\ar@{-->}[r]^\delta& \\
    E \ar[r]^{x'} & B' \ar[r]^{y'} & C \ar@{-->}[r]^{a_\star\delta}& .}
$$
By assumption, the bottom $\E$-triangle splits, and hence there exists $b'\in\C(B',E)$ such that $b'x'=\id_{E}$.
Thus we have that $(b'b)x=b'x'a=a$ and $E$ is injective.
\end{proof}

Let $(\C,\E,\s)$ be a triple satisfying (ET1) and (ET2). We say that it has \emph{enough injective objects} if for
any $A\in\C$, there exists an $\E$-triangle
$\xymatrix@C=0.5cm{
A \ar[r]^x & E \ar[r]^y & C\ar@{-->}[r]^\delta& }$ with $E$ an injective object.

\section{Phantom morphisms}

\subsection{Phantom morphisms in additive categories}

\begin{definition}\label{3.1}{\rm
Let $\F$ be an additive subfunctor of $\E$ and $\varphi\in \C(X,C)$. We call $\varphi$ an \emph{$\F$-phantom morphism} if
$\varphi^\star\delta\in \F(X,A)$ for any $\delta\in\E(C,A)$. Dually, let $\psi\in\C(A,Y)$.
We call $\psi$ an \emph{$\F$-cophantom morphism} if $\psi_\star\delta\in \F(C,Y)$ for any $\delta\in\E(C,A)$.
}
\end{definition}

We denote by $\textbf{Ph}(\F)$ and $\Coph(\F)$ the classes of $\F$-phantom and $\F$-cophantom morphisms respectively.
In this paper, we only discuss the properties of $\F$-phantom morphisms in most cases, but we need to keep in mind that
the dual results hold true for  $\F$-cophantom morphisms, and we will directly use it if necessary.

We first note that $\textbf{Ph}(\F)$ is an ideal.
Indeed, let $\varphi\in \C(X,C)$ be an $\F$-phantom morphism. If $f\in \C(X',X)$, then for any $\delta\in\E(C,A)$,
we have $(\varphi f)^\star\delta=f^\star(\varphi^\star\delta)$. Since $\varphi^\star\delta\in\F(X,A)$, we have
$(\varphi f)^\star\delta\in\F(X',A)$, and hence $\varphi f$ is an $\F$-phantom morphism. Similarly, let $g\in \C(A,A')$,
by the equality $(g\varphi)^\star\delta=\varphi^\star(g^\star\delta)$, we have $g\varphi$ is an $\F$-phantom morphism.
Moreover, if $\varphi_1,\varphi_2\in \C(X,C)$ are $\F$-phantom morphisms, then by the equality
$(\varphi_1+\varphi_2)^\star\delta={\varphi_1}^\star\delta+{\varphi_2}^\star\delta$, we have that $\varphi_1+\varphi_2$
is also an $\F$-phantom morphism. Therefore $\textbf{Ph}(\F)$ is an ideal.

\begin{example}
{\rm
\begin{itemize}
Let $R$ be a ring and $\E=\Ext^1_R(-,-)$
\item[(1)] If $\F=\mbox{PExt}^1_R(-,-)$ is as in Example \ref{2.8}(2),
then the $\F$-phantom morphism is the phantom morphism in \cite{He07} and the pure phantom morphism in \cite{FGHT}.
\item[(2)] If $\F=\mbox{GExt}^1_R(-,-)$ is as in Example \ref{2.8}(2), then the $\F$-(co)phantom  morphism is the
Gorenstein (co)phantom morphism in \cite{ZH16}. The following is a concrete example. Let $R=kQ/I$ with $k$ a field,
where $Q$ is the quiver
$$
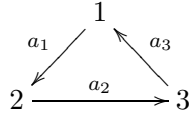
\xymatrix@R=20pt@C=20pt{
& 1\ar[ld]_{a_1}&\\
2\ar[rr]^{a_2}&&3\ar[lu]_{a_3}
}
$$
and $I=\langle a_1a_3a_2,a_2a_1a_3\rangle$. We can identify the irreducible Gorenstein cophantom morphisms
in the category of finite generated left $R$-modules as follows:
$$
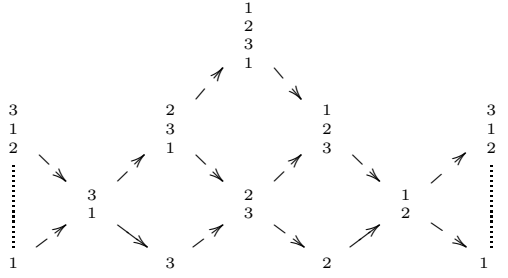
\xymatrix@R=5pt@C=10pt{&&&{\tiny\begin{array}{c}
      {1} \\ {2} \\ {3} \\ {1}
\end{array}}\ar@{-->}[dr]&&&\\
{\tiny\begin{array}{c}
     3\\ {1} \\ {2}
\end{array}}\ar@{-->}[dr]\ar@{.}[dd]&&{\tiny\begin{array}{c}
      {2} \\ {3} \\ {1}
\end{array}}\ar@{-->}[dr]\ar@{-->}[ur]&&{\tiny\begin{array}{c}
      {1} \\ {2} \\ {3}
\end{array}}\ar@{-->}[dr]&&{\tiny\begin{array}{c}
     3\\ {1} \\ {2}
\end{array}}\ar@{.}[dd]\\
&{\tiny\begin{array}{c} 3 \\ 1\end{array}}\ar@{-->}[ur]\ar[dr]&&{\tiny\begin{array}{c}
       {2} \\ {3}
\end{array}}\ar@{-->}[ur]\ar@{-->}[dr]&&{\tiny\begin{array}{c}    1 \\ 2 \end{array}}\ar@{-->}[ur]\ar@{-->}[dr]&\\
{\tiny\begin{array}{c}
      {1}
\end{array}}\ar@{-->}[ur]&&{\tiny \begin{array}{c}3\end{array}}\ar@{-->}[ur]&&
{\tiny
\begin{array}{c}2\end{array}}\ar[ur]&&{\tiny\begin{array}{c}
      {1}
\end{array}\ ,}
}
$$
where the morphisms marked by  the dashed arrows are all irreducible Gorenstein cophantom morphisms.
\end{itemize}}
\end{example}

Let $\I$ be an ideal of $\C$. We write
$$\I^\star:=\{i^\star\delta\mid i\in\I \mbox{ and }\delta \mbox{ is any $\E$-extension} \},$$
$$\I_\star:=\{i_\star\delta\mid i\in\I \mbox{ and }\delta \mbox{ is any $\E$-extension} \}.$$

\begin{proposition}\label{3.2}
$\I^\star$ is a minimal additive subfunctor of $\E$ for which $\I\subseteq\Ph(\I^\star)$.
\end{proposition}

\begin{proof}
We first prove that $\I^\star$ is an additive subfunctor of $\E$. Let $\varphi\in\C(X,C)$ and $\delta\in\I^\star(C,A)$,
that is, there exist $i\in\I(C,C')$ and $\delta'\in\E(C',A)$ such that $\delta=i^\star\delta'$. Then
$\varphi^\star\delta=\varphi^\star i^\star\delta'=(i\varphi)^\star\delta'$. Since $\I$ is an ideal of $\C$, we have
$i\varphi\in \I(X,C')$, and hence $\varphi^\star\delta\in \I^\star(X,A)$. Similarly, for $\psi\in \C(A,Y)$, by the equalities
$\psi_\star\delta=\psi_\star i^\star\delta'=i^\star\psi_\star\delta'$, we have $\psi_\star\delta\in\I^\star(C,Y)$.
The additivity of $\I^\star$ is induced by that of $\E$.

Next we prove the minimality about the property $\I\subseteq\Ph(\I^\star)$. Let $\F$ be any additive subfunctor of $\E$ satisfying
$\I\subseteq\Ph(\F)$. Let $\delta\in\I^\star(X,A)$, that is, there exist  $i\in\I(X,C)$ and $\delta'\in\E(C,A)$ such that
$\delta=i^\star\delta'$. Since $\I\subseteq\Ph(\F)$, we have $i\in\Ph(\F)$, and hence $\delta=i^\star\delta'\in\F(X,A)$.
This means that $\I^\star\subseteq\F$.
\end{proof}

Let $\mathcal{M}$ be a class of morphisms in $\C$. We write
$$\mathcal{M}^{\perp_\E}:=\{g\in\Mor\C\mid m^\star g_\star \delta=0
\mbox{ for any } m\in\mathcal{M}\mbox{ and any $\E$-extension } \delta \}.$$
Then $\mathcal{M}^{\perp_\E}$ is an ideal of $\C$. Indeed, if $g_1,g_2\in\mathcal{M}^{\perp_\E}$, then for any
$m\in\mathcal{M}$ and any $\E$-extension $\delta$, we have $m^\star (g_1+g_2)_\star \delta=
m^\star {g_1}_\star\delta+m^\star {g_2}_\star \delta=0$, which means that $g_1+g_2\in\mathcal{M}^{\perp_\E}$.
Let $g\in\mathcal{M}^{\perp_\E}$ and $h\in\Mor\C$ such that $hg$ is defined. Then $m^\star (hg)_\star \delta=
m^\star h_\star g_\star \delta= h_\star m^\star g_\star \delta=0$ implies that $hg\in\mathcal{M}^{\perp_\E}$.
Similarly, if $gk$ is defined for $k\in\Mor\C$, then $gk\in\mathcal{M}^{\perp_\E}$. Therefore $\mathcal{M}^{\perp_\E}$
is an ideal of $\C$.

Dually, we write
$${^{\perp_\E}\mathcal{M}}:=\{g\in\Mor\C\mid g^\star m_\star \delta=0
\mbox{ for any } m\in\mathcal{M}\mbox{ and any $\E$-extension } \delta \}.$$

\begin{definition}\label{3.3}{\rm
Let $f\in\C(X,C)$ and $g\in\C(A,Y)$. The pair $(f,g)$ is said to be \emph{$\E$-orthogonal} if $f^\star g_\star\delta=0$
(or equivalently, $g_\star f^\star \delta=0$) for any $\delta\in\E(C,A)$.}
\end{definition}

We write
$$\F\mbox{-}{\inj}:=\{i\in\Mor\C\mid i_\star\delta=0 \mbox{ for each $\F$-extension }\delta\},$$
and call the elements in $\F\mbox{-}{\inj}$ \emph{$\F$-injective morphisms}. Dually, we write
$$\F\mbox{-}{\proj}:=\{i\in\Mor\C\mid i^\star\delta=0 \mbox{ for each $\F$-extension }\delta\},$$
and call the elements in $\F\mbox{-}{\proj}$ \emph{$\F$-projective morphisms}.

\begin{proposition}\label{3.4}
\begin{enumerate}
\item[]
\item[(1)] The pair $(\Ph(\F),\F\mbox{-}\inj)$ is $\E$-orthogonal.
\item[(2)] Let $\I$ be an ideal of $\C$. Then $\I^\star\mbox{-}\inj=\I^{\perp_\E}$.
\end{enumerate}
\end{proposition}

\begin{proof}
(1) It is clear.

(2) Let $j\in\I^\star\mbox{-}\inj$. For any $i\in\I$ and any $\E$-extension $\delta$, we have that $i^\star\delta$
is an $\I^\star$-extension. So $i^\star j_\star\delta=j_\star i^\star\delta=0$ and $j\in \I^{\perp_\E}$.
Conversely, let $j\in \I^{\perp_\E}$. For any $\I^\star$-extension $\delta$, there exist $i\in\I$ and an $\E$-extension
$\delta'$ such that $\delta=i^\star\delta'$. So $j_\star\delta=j_\star i^\star\delta'=i^\star j_\star \delta'=0$ and
$j\in\I^\star\mbox{-}\inj$.
\end{proof}

\subsection{Phantom morphisms in extriangulated categories}

Let $(\C,\E,\s)$ be a triple satisfying (ET1) and (ET2), and let
$\xymatrix@C=0.5cm{
A \ar[r]^x & B \ar[r]^y & C\ar@{-->}[r]^\delta& }$ be any $\E$-triangle. Then for a morphism $\varphi\in\C(X,C)$,
there exists a morphism of  $\E$-triangles
$$\xymatrix{
A \ar[r]^{x'}\ar@{=}[d] & B' \ar[r]^{y'}\ar[d]^g & X\ar[d]^\varphi\ar@{-->}[r]^{\varphi^\star\delta}& \\
A \ar[r]^{x} & B \ar[r]^{y} & C \ar@{-->}[r]^{\delta}& .}$$
We easily see that $\varphi\in\C(X,C)$ is an $\F$-phantom morphism if and only if every $\E$-triangle
$\xymatrix@C=0.5cm{
A \ar[r]^{x'}& B' \ar[r]^{y'} & X\ar@{-->}[r]^{\varphi^\star\delta}&}$ induced by the above is an $\F$-triangle.

Now let $(\C,\E,\s)$ be a triple satisfying (ET1), (ET2) and ${\rm (ET3)^{op}}$, and let
$\xymatrix@C=0.5cm{A \ar[r]^x & B \ar[r]^y & C\ar@{-->}[r]^\delta& }$ be any $\E$-triangle and $\varphi\in\C(X,C)$ an
$\F$-phantom morphism. For any $\F$-projective morphism $p\in\C(P,X)$, we have a morphism of $\E$-triangles
$$\xymatrix{
A \ar[r]^{}\ar@{=}[d] & Q \ar[r]^{}\ar[d] & P\ar[d]^p\ar@{-->}[r]^{{p^\star\varphi^\star\delta}}& \\
A \ar[r]^{x'} & B' \ar[r]^{y'} & X \ar@{-->}[r]^{\varphi^\star\delta}&. }$$
Since $\varphi^\star\delta$ is an $\F$-extension and $p$ is an $\F$-projective morphism, we have ${p^\star\varphi^\star\delta}=0$,
and hence there exists $p'\in\C(P,B')$ such that $p=y'p'$.
Then $ygp'=\varphi y'p'=\varphi p$ and we get the following commutative diagram
$$\xymatrix{
0 \ar[r]^{} & P \ar[r]^{\id_P}\ar[d]^{gp'} & P\ar[d]^{\varphi p}\ar@{-->}[r]^{0}& \\
A \ar[r]^{x} & B \ar[r]^{y} & C \ar@{-->}[r]^{\delta}&. }$$
By ${\rm (ET3)^{op}}$, we have a morphism of $\E$-triangles
$$\xymatrix{
0 \ar[r]^{}\ar[d] & P \ar[r]^{\id_P}\ar[d]^{gp'} & P\ar[d]^{\varphi p}\ar@{-->}[r]^{0}& \\
A \ar[r]^{x} & B \ar[r]^{y} & C \ar@{-->}[r]^{\delta}&. }$$
In particular, $(\varphi p)^\star\delta=0$, that is, the composition $\xymatrix@C=0.5cm{
  P \ar[r]^p & X \ar[r]^\varphi & C }$ is an $\E$-projective morphism.

Therefore, if we consider the stable category $\underline{(\C,\E,\s)}:=(\C,\E,\s)\diagup\E\mbox{-}\proj$,
where the objects in $\underline{(\C,\E,\s)}$ are the objects in $\C$, and for any $X,Y\in\C$, the morphism set
$\Hom(X,Y)$ in $\underline{(\C,\E,\s)}$ are the morphism set $\C(X,Y)\diagup\E\mbox{-}\proj$, then
$\F$-phantom morphisms make $\F$-projective morphisms vanish in $\underline{(\C,\E,\s)}$.
This is also why we call these morphisms ``$\F$-phantom''  on some level.

\begin{definition}\label{3.5}
{\rm Let $\I$ be an ideal of $\C$ and $C\in\C$.
\begin{enumerate}
\item[(1)] An \emph{$\I$-precover} of $C$ is a morphism $i:X\ra C$ in $\I$ such that any morphism
$i':X'\ra C$ in $\I$ factors through $i$, that is, there exists a morphism $g:X'\ra X$ such that $i'=ig$.
$$\xymatrix{& X'\ar[d]^{i'}\ar@{-->}[ld]_g \\X\ar[r]^i & C.}$$
\item[(2)] Let $(\C,\E,\s)$ be a triple satisfying (ET1) and (ET2). A morphism  $i:X\ra C$ in $\I$ is called
a \emph{special $\I$-precover} of $C$ if there exists a morphism of $\E$-triangles
$$\xymatrix{
A \ar[r]\ar[d]_j & B \ar[r]\ar[d] & C\ar@{=}[d]\ar@{-->}[r]^\delta& \\
A' \ar[r] & X \ar[r]^{i} & C \ar@{-->}[r]^{\delta'}&}$$
with $j\in\I^{\perp_\E}$.
\end{enumerate}}
\end{definition}

An ideal $\I$ of $\C$ is called a {\it (special) precovering ideal}
of $\mathscr{C}$ if any object in $\mathscr{C}$ admits an (a special) $\I$-precover.
Dually, the notions of a {\it (special) $\I$-preenvelope} and
a {\it (special) preenveloping ideal} are defined.

In what follows, we always assume that the triple $(\C,\E,\s)$  satisfies (ET1) and (ET2).

\begin{proposition}\label{3.6}
Every special $\I$-precover is an $\I$-precover.
\end{proposition}

\begin{proof}
Let $C\in\C$, and $i:X\ra C$  is a special $\I$-precover of $C$. Then there exists a morphism of $\E$-triangles
$$\xymatrix{
A \ar[r]\ar[d]_j & B \ar[r]\ar[d] & C\ar@{=}[d]\ar@{-->}[r]^\delta& \\
A' \ar[r] & X \ar[r]^{i} & C \ar@{-->}[r]^{j_\star\delta}& }$$
with $j\in\I^{\perp_\E}$. Now for any $i':X'\ra C$ in $\I$, there exists a morphism of $\E$-triangles
$$\xymatrix{
A' \ar[r]\ar@{=}[d] & Y \ar[r]\ar[d]^k & X'\ar[d]^{i'}\ar@{-->}[r]^{{i'}^\star j_\star\delta}& \\
A' \ar[r] & X \ar[r]^i & C \ar@{-->}[r]^{j_\star\delta}& .}$$
Since $i'\in\I$ and $j\in\I^{\perp_\E}$, we have ${i'}^\star j_\star\delta=0$. So the sequence $\xymatrix@C=0.5cm{
A' \ar[r] & Y \ar[r] & X'}$ splits and there exists $g:X'\ra Y$ such that $i'=ikg$. It follows that
$i:X\ra C$  is an $\I$-precover of $C$.
\end{proof}

\begin{definition}\label{3.7}{\rm
An $\E$-orthogonal pair $(\I,\J)$ of ideals of $\C$ is called an \emph{$\E$-cotorsion pair} if
$\I={^{\perp_\E}\J}$ and $\J={\I^{\perp_\E}}$.}
\end{definition}

The following result gives a sufficient condition such that an $\E$-orthogonal pair of ideals is an $\E$-cotorsion pair.

\begin{theorem}\label{3.8}
If $\I$ is a special precovering ideal, then the pair $(\I,{\I^{\perp_\E}})$ of ideals is an $\E$-cotorsion pair.
\end{theorem}

\begin{proof}
Clearly, $\I\subseteq{^{\perp_\E}({\I^{\perp_\E}})}$. Now let $i' \in {^{\perp_\E}({\I^{\perp_\E}})}$ with
$i':X'\ra C$. For the object $C$, take a special $\I$-precover $i:X\ra C$. Then
there exists a morphism of $\E$-triangles
$$\xymatrix{
A \ar[r]\ar[d]_j & B \ar[r]\ar[d] & C\ar@{=}[d]\ar@{-->}[r]^\delta& \\
A' \ar[r] & X \ar[r]^{i} & C \ar@{-->}[r]^{j_\star\delta}& }$$
with $j\in\I^{\perp_\E}$. Furthermore, assume that $\s({{i'}^\star j_\star\delta})=[\xymatrix@C=0.5cm{
A' \ar[r]^x & Y \ar[r]^y & X'}]$. Then we also have  a morphism of $\E$-triangles
$$\xymatrix{
A' \ar[r]^x\ar@{=}[d] & Y \ar[r]^y\ar[d]^k & X'\ar[d]^{i'}\ar@{-->}[r]^{{i'}^\star j_\star\delta}& \\
A' \ar[r] & X \ar[r]^i & C \ar@{-->}[r]^{j_\star\delta}& .}$$
Since $i'\in {^{\perp_\E}({\I^{\perp_\E}})}$ and $j\in\I^{\perp_\E}$, we have ${{i'}^\star j_\star\delta}=0$,
and hence there exists $y':X'\ra Y$ such that $i'=i(ky')$. Thus we have that $i'\in \I$ and
${^{\perp_\E}({\I^{\perp_\E}})}\subseteq\I$. Therefore $\I={^{\perp_\E}({\I^{\perp_\E}})}$ and
$(\I,{\I^{\perp_\E}})$ is an $\E$-cotorsion pair.
\end{proof}

\begin{corollary}\label{3.9}
If $\I$ is a special precovering ideal, then $\I=\Ph(\I^{\star})$.
\end{corollary}

\begin{proof}
By definition, we have $\I\subseteq\Ph(\I^{\star})$. Now let $\varphi\in\Ph(\I^{\star})$, that is,
$\varphi$ is an $\I^{\star}$-phantom morphism.
Then for any $\E$-extension $\delta$, we have that $\varphi^\star\delta$ is an $\I^\star$-extension.
Let $j\in{\I^{\perp_\E}}$. Since ${\I^{\perp_\E}}=\I^\star\mbox{-}\inj$ by Proposition \ref{3.4}(2), we have
$j\in\I^\star\mbox{-}\inj$. So $\varphi^\star j_\star\delta=j_\star\varphi^\star\delta=0$ and
$\varphi\in{^{\perp_\E}({\I^{\perp_\E}})}$. Furthermore, $\I={^{\perp_\E}({\I^{\perp_\E}})}$ by Theorem \ref{3.8}.
So $\varphi\in\I$ and $\Ph(\I^{\star})\subseteq\I$.
\end{proof}

From Proposition \ref{3.4}(1), we have known that $(\Ph(\F),\F\mbox{-}\inj)$ is an $\E$-orthogonal pair.
In the rest of this section, we mainly study when it is an $\E$-cotorsion pair. To do it, we first introduce
the following

\begin{definition}\label{3.10}
{\rm \begin{enumerate}
\item[]
\item[(1)] An additive subfunctor $\F$ of $\E$ is said to \emph{have enough injective morphisms}
if for any $A\in\C$, there exists an $\F$-triangle
$\xymatrix@C=0.5cm{A \ar[r]^e & B \ar[r] & C\ar@{-->}[r]^\delta& },$
where $e$ is an $\F$-injective morphism.
\item[(2)] The additive subfunctor $\F$ of $\E$ is said to \emph{have enough special injective morphisms}
if for any $A\in\C$, there exists an $\F$-triangle as above together with a morphism of $\E$-triangles
    $$
\xymatrix{
   A \ar[r]^e\ar@{=}[d] & B \ar[r]\ar[d] & C\ar[d]^{\varphi}\ar@{-->}[r]^{\delta}& \\
    A \ar[r] & B' \ar[r]^i & C' \ar@{-->}[r]^{\delta'}&}
$$
with $\varphi$ an $\F$-phantom morphism.
\end{enumerate}
}
\end{definition}\label{3.11}

\begin{lemma}{
Let $(\C,\E,\s)$ be a triple satisfying (ET1), (ET2) and (ET3). If an $\F$-inflation $x:A\ra B$ factors
through an $\E$-inflation $g:A\ra Y$, then $g$ is an $\F$-inflation.
}
\end{lemma}

\begin{proof}
Since $x:A\ra B$ is an $\F$-inflation, there exists an $\F$-triangle
$\xymatrix@C=0.5cm{A \ar[r]^x & B \ar[r] & C\ar@{-->}[r]^\delta& }$;
since $g:A\ra Y$ is an $\E$-inflation, there exists an $\E$-triangle
$\xymatrix@C=0.5cm{
A \ar[r]^g & Y \ar[r] & Z\ar@{-->}[r]^{\delta'}& }$ together with the following commutative diagram
$$\xymatrix{
A \ar[r]^g\ar@{=}[d] & Y \ar[r]\ar[d] & Z\ar@{-->}[r]^{\delta'}& \\
A \ar[r]^x & B \ar[r] & C \ar@{-->}[r]^{\delta}& .}$$
By (ET3), we get a morphism of $\E$-triangles
$$\xymatrix{
A \ar[r]^g\ar@{=}[d] & Y \ar[r]\ar[d] & Z\ar[d]^{h}\ar@{-->}[r]^{\delta'}& \\
A \ar[r]^x & B \ar[r] & C \ar@{-->}[r]^{\delta}&. }$$
In particular, we have $\delta'=h^\star\delta$. So $\delta'$ is an $\F$-extension and $g$ is an $\F$-inflation.
\end{proof}

\begin{proposition}\label{3.12}
Let $(\C,\E,\s)$ be a triple satisfying (ET1), (ET2) and (ET3). If $\F\subseteq\E$ is an additive
subfunctor having enough injective morphisms, then $\Ph(\F)={^{\perp_\E}(\F\mbox{-}\inj)}$.
\end{proposition}

\begin{proof}
By Proposition \ref{3.4}(1), we have $\Ph(\F)\subseteq{^{\perp_\E}(\F\mbox{-}\inj)}$.

Now let $f:X\ra C\in{^{\perp_\E}(\F\mbox{-}\inj)}$, and let
$\xymatrix@C=0.5cm{A \ar[r] & B \ar[r] & C\ar@{-->}[r]^\delta&}$ be any $\E$-triangle.
Then we have a morphism of $\E$-triangles
$$\xymatrix{
A \ar[r]^i\ar@{=}[d] & B' \ar[r]\ar[d] & X\ar[d]^{f}\ar@{-->}[r]^{f^\star\delta}& \\
A \ar[r] & B \ar[r] & C \ar@{-->}[r]^{\delta}& .}$$
For the object $A$, by assumption there exists an $\F$-injective $\F$-inflation $e:A\ra Y$.
Consider the following morphism of $\E$-triangles
$$\xymatrix{A \ar[r]^i\ar[d]^e & B' \ar[r]\ar[d]^g & X\ar@{=}[d]\ar@{-->}[r]^{f^\star\delta}& \\
Y \ar[r] & Z \ar[r] & X \ar@{-->}[r]^{e_\star f^\star\delta}& .}$$
Since $e_\star f^\star\delta =f^\star e_\star \delta=0$, that is, the $\E$-triangle
$\xymatrix@C=0.5cm{Y \ar[r] & Z \ar[r] & X \ar@{-->}[r]^{e_\star f^\star\delta}& }$ splits,
there exists $h:Z\ra Y$ such that $e=(hg)i$. By Lemma \ref{3.11}, $i$ is also an $\F$-inflation. So each
$\xymatrix@C=0.5cm{A \ar[r]^i & B' \ar[r] & X\ar@{-->}[r]^{f^\star\delta}&}$ induced by any $\E$-triangle
along $f$ is an $\F$-triangle, which implies that $f$ is an $\F$-phantom morphism. Thus
${^{\perp_\E}(\F\mbox{-}\inj)}\subseteq\Ph(\F)$, and therefore $\Ph(\F)={^{\perp_\E}(\F\mbox{-}\inj)}$.
\end{proof}

Note that a morphism $e:A\ra X$ in $\I$ is called a \emph{special $\I$-preenvelope} of $A$
if there exists a morphism of $\E$-triangles
$$\xymatrix{A \ar[r]^e\ar@{=}[d] & X \ar[r]\ar[d] & Y\ar[d]^{j}\ar@{-->}[r]^{\delta}& \\
A \ar[r] & B \ar[r] & C \ar@{-->}[r]^{\delta'}& }$$
with $j\in{^{\perp_\E}\I}$.

Now if $\F$ has enough special injective morphisms, then for any $A\in\C$, there exists an $\F$-triangle
$\xymatrix@C=0.5cm{A \ar[r]^e & X \ar[r] & Y\ar@{-->}[r]^{\delta}&}$
together with a morphism of $\E$-triangles
$$\xymatrix{
A \ar[r]^e\ar@{=}[d] & X \ar[r]\ar[d] & Y\ar[d]^{j}\ar@{-->}[r]^{\delta}& \\
A \ar[r] & B \ar[r]^i & C' \ar@{-->}[r]^{\delta'}& }$$
with $e\in\F\mbox{-}\inj$ and $j\in\Ph(\F)$. By Proposition \ref{3.12}, we have
$\Ph(\F)={^{\perp_\E}(\F\mbox{-}\inj)}$. So $j\in {^{\perp_\E}(\F\mbox{-}\inj)}$ and $e$ is a special
$\F$-injective preenvelope of $A$. This shows that $\F\mbox{-}\inj$ is a special preenveloping ideal.

As a dual of Theorem \ref{3.8}, we have the following

\begin{theorem}\label{3.13}
If $\J$ is a special preenveloping ideal of $\C$, then the orthogonal pair $({^{\perp_\E}\J}, \J)$
of ideals is an $\E$-cotorsion pair.
\end{theorem}

Note that if $\F$ has enough special injective morphisms, then $({^{\perp_\E}(\F\mbox{-}\inj)},\F\mbox{-}\inj)$
is an $\E$-cotorsion pair of ideals by Theorem \ref{3.13}. Because $\Ph(\F)={^{\perp_\E}(\F\mbox{-}\inj)}$ by
Proposition \ref{3.12}, we get the following

\begin{corollary}\label{3.14}
Let $(\C,\E,\s)$ be a triple satisfying (ET1), (ET2) and (ET3). If $\F$ has enough special injective morphisms,
then $(\Ph(\F),\F\mbox{-}\inj)$ is an $\E$-cotorsion pair of ideals; in particular, $\Ph(\F)^{\perp_\E}=\F\mbox{-}\inj$.
\end{corollary}

\section{The interplay between phantom ideals and cotorsion pairs}

From the previous section, we know that a special precovering ideal corresponds an $\E$-cotorsion pair, and that a phantom ideal
induced by a subfunctor also corresponds ones under a suitable assumption. In this section, we will investigate their interplay by
showing that a phantom ideal induced by a subfunctor is a  special precovering ideal under some suitable assumption, and vice versa.
Before doing it, we first give the following lemma, which simplifies the calculation process for checking phantom morphisms.

\begin{lemma}\label{4.1}
Let $(\C,\E,\s)$ be an extriangulated category. Consider an $\E$-triangle $\xymatrix@C=0.5cm{
K \ar[r] & P \ar[r]^p & C\ar@{-->}[r]^{\gamma}&}$ with $p$ an $\E$-projective morphism and a morphism $\varphi:X\ra C$.
Then the following statements are equivalent.
\begin{enumerate}
\item[(1)] $\varphi$ is an $\F$-phantom morphism.
\item[(2)] The induced $\E$-triangle $\xymatrix@C=0.5cm{
K \ar[r] & Y \ar[r] & X\ar@{-->}[r]^{\varphi^\star\gamma}&}$ is an $\F$-triangle.
\end{enumerate}
\end{lemma}

\begin{proof}
(1) $\Rightarrow$ (2) It is trivial.

(2) $\Rightarrow$ (1) Let $\xymatrix@C=0.5cm{
A \ar[r] & B \ar[r]^y & C\ar@{-->}[r]^{\delta}&}$ be any $\E$-triangle. By \cite[Proposition 3.15]{NP},
we have the following commutative diagram
$$\xymatrix{&K\ar@{=}[r]\ar[d]&K\ar[d]&\\
A\ar[r]\ar@{=}[d]&Q\ar[r]\ar[d]&P\ar[d]^p\ar@{-->}[r]^{\delta'}&\\
A \ar[r] & B \ar[r]^y \ar@{-->}[d]^{\gamma'}& C\ar@{-->}[r]^{\delta}\ar@{-->}[d]^{\gamma}&\\
&&&}$$
in $\C$ with $\delta'=p^\star\delta$ and $\gamma'=y^\star\gamma$. Since $p$ is an $\E$-projective morphism, we have $\delta'=p^\star\delta=0$,
and hence the middle row splits. Then there exists $g:P\ra B$ such that $p=yg$, that is, the following diagram
$$\xymatrix{
K \ar[r] & P \ar[r]^p\ar[d]^g & C\ar@{=}[d]\ar@{-->}[r]^{\gamma}& \\
A \ar[r] & B \ar[r]^y & C \ar@{-->}[r]^{\delta}& }$$
is commutative. By ${\rm (ET3)^{op}}$, there exists a morphism of $\E$-triangles
$$\xymatrix{
K \ar[r]\ar[d]^f & P \ar[r]^p\ar[d]^g & C\ar@{=}[d]\ar@{-->}[r]^{\gamma}& \\
A \ar[r] & B \ar[r]^y & C \ar@{-->}[r]^{\delta}& .}$$
In particular, we have $\delta=f_\star\gamma$. Thus $\varphi^\star\delta=\varphi^\star f_\star\gamma=f_\star\varphi^\star \gamma$.
By assumption, $\varphi^\star \gamma$ is an $\F$-extension, and hence $\varphi^\star\delta$ is also an $\F$-extension,
which shows that $\varphi:X\ra C$ is an $\F$-phantom morphism.
\end{proof}

Now we show that, under a suitable assumption, phantom ideals induced by additive subfunctors having enough injective morphisms
are special precovering ideals.

\begin{theorem}\label{4.2}
Let $(\C,\E,\s)$ be an extriangulated category with enough projective morphisms, and assume that $\F\subseteq\E$ is an
additive subfunctor having enough injective morphisms. Then $\Ph(\F)$ is a special precovering ideal.
\end{theorem}

\begin{proof}
Let $C\in\C$. Then by assumption, there exists an $\E$-triangle $\xymatrix@C=0.5cm{
K \ar[r] & P \ar[r]^p & C\ar@{-->}[r]^{\gamma}&}$ with $p$ an $\E$-projective morphism.
For the object $K$, there exists an $\F$-injective $\F$-inflation $e:K\ra X$. Then we get a morphism of $\E$-triangles
$$\xymatrix{
K \ar[r]\ar[d]^e & P \ar[r]^p\ar[d] & C\ar@{=}[d]\ar@{-->}[r]^{\gamma}& \\
X \ar[r] & Y \ar[r]^\varphi & C \ar@{-->}[r]^{e_\star\gamma}& .}$$
In the following, we argue that $\varphi$ is a special $\Ph(\F)$-precover of $C$.

First of all, by Proposition \ref{3.4}(1), we have that $\F\mbox{-}\inj\subseteq\Ph(\F)^{\perp_\E}$
and $e\in \Ph(\F)^{\perp_\E}$. Moreover, consider the following diagram of morphisms of $\E$-triangles
$$\xymatrix{
K \ar[r]^i\ar@{=}[d] & Z \ar@{.>}[ldd]\ar[r]\ar[d] & Y\ar@{.>}[ldd]\ar[d]^{\varphi}\ar@{-->}[r]^{\varphi^\star\gamma}& \\
K \ar[r]\ar[d]_e & P \ar[r]\ar[d] & C\ar@{=}[d]\ar@{-->}[r]^{\gamma}& \\
X \ar[r] & Y \ar[r]^\varphi & C \ar@{-->}[r]^{e_\star\gamma}&. }$$
Since $\id_C\varphi=\varphi=\varphi\id_Y$, there exists $z:Z\ra C$ such that $e=e\id_K=zi$ by \cite[Corollary 3.5]{NP}.
Since $e$ is an $\F$-inflation, $i$ is also an $\F$-inflation by Lemma \ref{3.11}, and hence $\varphi^\star\gamma$ is an $\F$-extension.
By Lemma \ref{4.1}, $\varphi$ is an $\F$-phantom morphism and it is a special $\Ph(\F)$-precover of $C$.

Therefore we conclude that $\Ph(\F)$ is a special precovering ideal.
\end{proof}

\begin{lemma}\label{4.3}
Let $\mathcal{M}$ be a class of morphisms in $\C$. Consider a morphism of $\E$-triangles
$$
\xymatrix{
   A \ar[r]^a\ar[d]^f & B \ar[r]^b\ar[d]^g & E\ar[d]^h\ar@{-->}[r]^{\gamma}& \\
    X \ar[r]^x & Y \ar[r]^y & Z \ar@{-->}[r]^{\delta}& .}
$$
If $f\in\mathcal{M}^{\perp_\E}$ and $E$ is an injective object, then $g\in\mathcal{M}^{\perp_\E}$.
\end{lemma}

\begin{proof}
Since $\mathcal{M}^{\perp_\E}$ is an ideal and $f\in\mathcal{M}^{\perp_\E}$, we have $ga=xf\in\mathcal{M}^{\perp_\E}$.
Thus for any $m\in\mathcal{M}$, we have $m^\star g_\star a_\star=m^\star(ga)_\star=0$. On one hand, by Lemma \ref{2.11}
we have that an object $E\in\C$ is injective if and only if $\E(C,E)=0$ for any $C\in\C$. On the other hand,
by \cite[Corollary 3.12]{NP}, there exists an exact sequence
$$\xymatrix@C=0.5cm{\E(C,A) \ar[r]^{a_\star} & \E(C,B) \ar[r] & \E(C,E) }.$$
Thus $a_\star$ is epic and $m^\star g_\star=0$, which shows that $g\in\mathcal{M}^{\perp_\E}$.
\end{proof}

\begin{theorem}\label{4.4}
Let $(\C,\E,\s)$ be an extriangulated category with enough injective objects. If $\I$ is a special precovering ideal,
then $\I^{\perp_\E}$ is a special preenveloping ideal.
\end{theorem}

\begin{proof}
Let $A\in \C$ and $\xymatrix@C=0.5cm{
A\ar[r] & E \ar[r]^c & C\ar@{-->}[r]^{\delta}&}$ be an $\E$-triangle with $E$ an injective object.
For the object $C$, there exists a special $\I$-precover $x:X\ra C$. Then we have a morphism of $\E$-triangles
$$\xymatrix{
A \ar[r]^a\ar@{=}[d] & B \ar[r]\ar[d] & X\ar[d]^{x}\ar@{-->}[r]^{x^\star\delta}& \\
A \ar[r] & E \ar[r]^c & C \ar@{-->}[r]^{\delta}&. }$$
In the following, we argue that $a$ is a special $\I^{\perp_\E}$-preenvelope of $A$.

Since $x\in\I\subseteq{^{\perp_\E}(\I^{\perp_\E})}$, it suffices to show that $a\in\I^{\perp_\E}$.
Assume that the special $\I$-precover $x:X\ra C$ comes from the following morphism of $\E$-triangles
$$\xymatrix{
Y \ar[r]^y\ar[d]_g & Z \ar[r]^z\ar[d]_h & C\ar@{=}[d]\ar@{-->}[r]^{\gamma}& \\
W \ar[r]^w & X \ar[r]^x & C \ar@{-->}[r]^{g_\star\gamma}& }$$
with $g\in \I^{\perp_\E}$. Consider the following commutative diagram
$$\xymatrix{
&&&Y\ar@{=}[rr]\ar[ld]_g\ar'[d][dd]&&Y\ar[ld]_g\ar[dd]&\\
&&W\ar@{=}[rr]\ar[dd]&&W\ar[dd]&&\\
& A \ar'[r][rr]^{e'} \ar@{=}[ld]\ar@{=}'[d][dd]
&  & F \ar[ld]_k\ar'[d][dd] \ar'[r][rr] &&Z\ar[ld]_h\ar@{-->}[rr]^{~~ z^\star\delta=h^\star x^\star\delta}\ar'[d][dd]^z && \\
A\ar[rr]^{~~~a}\ar@{=}[dd]&&B\ar[rr]\ar[dd]&&X\ar@{-->}[rr]^{x^\star\delta}\ar[dd]_{x}&&&\\
& A \ar@{=}[ld]\ar'[r][rr]
&  & E  \ar@{-->}'[d][dd]^{c^\star\gamma}\ar@{=}[ld] \ar'[r][rr]^{} &&C\ar@{-->}'[d][dd]^\gamma\ar@{=}[ld]\ar@{-->}[rr]^\delta && \\
 A \ar[rr]^{}
&& E\ar@{-->}[dd]_{c^\star g_\star\gamma=g_\star c^\star \gamma} \ar[rr]^{~~~~c}
&&C\ar@{-->}[dd]^{g_\star \gamma}\ar@{-->}[rr]^{~~~~~~\delta}&&&\\
&&&&&.&\\ &&&&&&}$$
By Lemma \ref{4.3} and the vertical plane in the middle of the above diagram, we have
$k\in\I^{\perp_\E}$, and hence $a=ke'\in\I^{\perp_\E}$, as desired.
\end{proof}

Following the above theorem and its dual, we get a morphism version of the Salce's lemma as follows.

\vspace{0.2cm}

{\bf Salce's Lemma.}   Let $(\C,\E,\s)$ be an extriangulated category with enough projective and injective objects.
If $(\I,\J)$ is an $\E$-cotorsion pair of ideals, then
$\I$ is a special precovering ideal if and only if $\J$ is a special preenveloping ideal.

\vspace{0.2cm}

Now we give our main result as follows. Here, an $\E$-cotorsion pair $(\I,\J)$ of ideals is called \emph{complete} if
$\I$ is a special precovering ideal and $\J$ is a special preenveloping ideal.

\begin{theorem}\label{4.5}
Let  $(\C,\E,\s)$ be an extriangulated category. Then we have the following implications.
$$
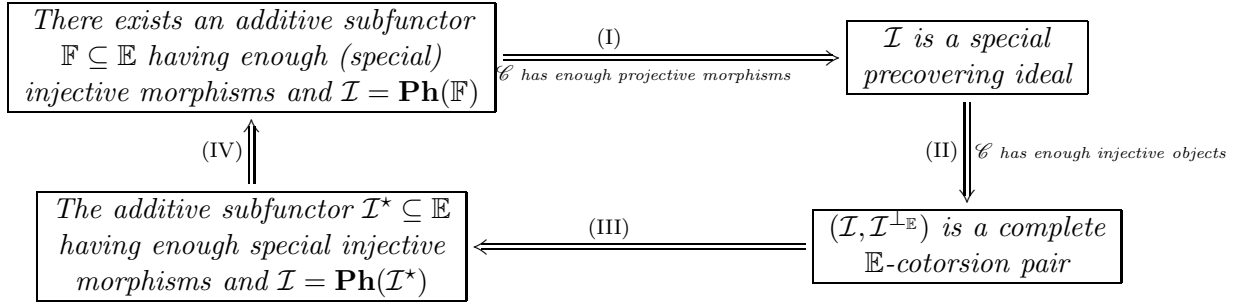
\xymatrix{
{\begin{tabular}{|c|}
  \hline
      \mbox{There exists an additive subfunctor} \\ \mbox{$\F\subseteq\E$ having enough (special)}\\\mbox{injective morphisms and $\I=\Ph(\F)$} \\
  \hline
\end{tabular}}\ar@{=>}[rrrr]^{\rm(I)}_{\tiny ~~~~~~~~~~ \C\mbox{ has enough projective morphisms}}&&&&{\begin{tabular}{|c|}
  \hline
        \mbox{$\I$ is a special}\\\mbox{precovering ideal} \\
  \hline
\end{tabular}}\ar@{=>}[d]_{\rm(II)}^{\tiny \C \mbox{ has enough injective objects}}\\
{\begin{tabular}{|c|}
  \hline
      \mbox{The additive subfunctor $\I^\star\subseteq\E$} \\ \mbox{having enough special injective}\\\mbox{morphisms and $\I=\Ph(\I^\star)$} \\
  \hline
\end{tabular}}\ar@{=>}[u]^{\rm(IV)}&&&&{\begin{tabular}{|c|}
  \hline
      \mbox{$(\I,\I^{\perp_\E})$ is a complete} \\ \mbox{$\E$-cotorsion pair}\\
  \hline
\end{tabular}}\ar@{=>}[llll]_{\rm(III)}}$$
\end{theorem}

\begin{proof}
(I) It follows directly from Theorem \ref{4.2}.

(II) Since $\I$ is a special precovering ideal, $(\I,\I^{\perp_\E})$ is an $\E$-cotorsion pair by Theorem \ref{3.8}.
Moreover, since $\C$ has enough injective objects, $\I^{\perp_\E}$ is a special preenveloping ideal by Theorem \ref{4.4}.
Thus $(\I,\I^{\perp_\E})$ is a complete $\E$-cotorsion pair.

(III) First, since $\I$ is a special precovering ideal, we have $\I=\Ph(\I^\star)$ by Corollary \ref{3.9}. Moreover,
we have $\I^{\perp_\E}=\I^\star\mbox{-}\inj$ by Proposition \ref{3.4}(2).
So by assumption, any object in $\C$ admits a special $\I^\star$-injective preenvelope, that is, for any $A\in\C$, there exists
an $\I^\star$-injective morphism $e:A\ra X$ that comes from a morphism of $\E$-triangles
$$\xymatrix{A \ar[r]^e\ar@{=}[d] & X \ar[r]\ar[d] & Y\ar[d]^{j}\ar@{-->}[r]^{j^\star\delta}& \\
A \ar[r] & B \ar[r] & C \ar@{-->}[r]^{\delta}& }$$ with $j\in {^{\perp_\E}(\I^\star\mbox{-}\inj)}$.
This, on the other hand, shows that $\I^\star$ has enough injective morphisms. So by Proposition \ref{3.12}, we have that
${^{\perp_\E}(\I^\star\mbox{-}\inj)}=\Ph(\I^\star)$ and $j\in\Ph(\I^\star)$. It follows that $\I^\star$ has enough special injective morphisms.

(IV) It is trivial.
\end{proof}

By Theorem \ref{4.5}, we have that if $(\C,\E,\s)$ is an extriangulated category with enough injective objects and projective morphisms,
then we get the following bijective correspondence.
$$(\bigstar) \ \
\xymatrix@C=1.5cm{
{\begin{tabular}{|c|}
\hline
\mbox{all special precovering ideals} \\ \mbox{of $\C$}\\
\hline
\end{tabular}}\ar@<+4pt>[r]^{\!\!\!\!\!\!\!\!\!\!\!\!\!\!\!(-)^\star}&\ar@<+4pt>[l]^{\!\!\!\!\!\!\!\!\!\!\!\!\!\!\!\!\Ph(-)}{\begin{tabular}{|c|}
\hline
\mbox{all additive subfunctors of $\E$ having}\\\mbox{enough special injective morphisms} \\
\hline
\end{tabular}}}$$

Combining it with the Salce's lemma, we further get the following

\begin{theorem}
Let $(\C,\E,\s)$ be an extriangulated category with enough injective objects and projective objects.
Then we have the following implications for an $\E$-cotorsion pair $(\I,\J)$ of ideals.
$$
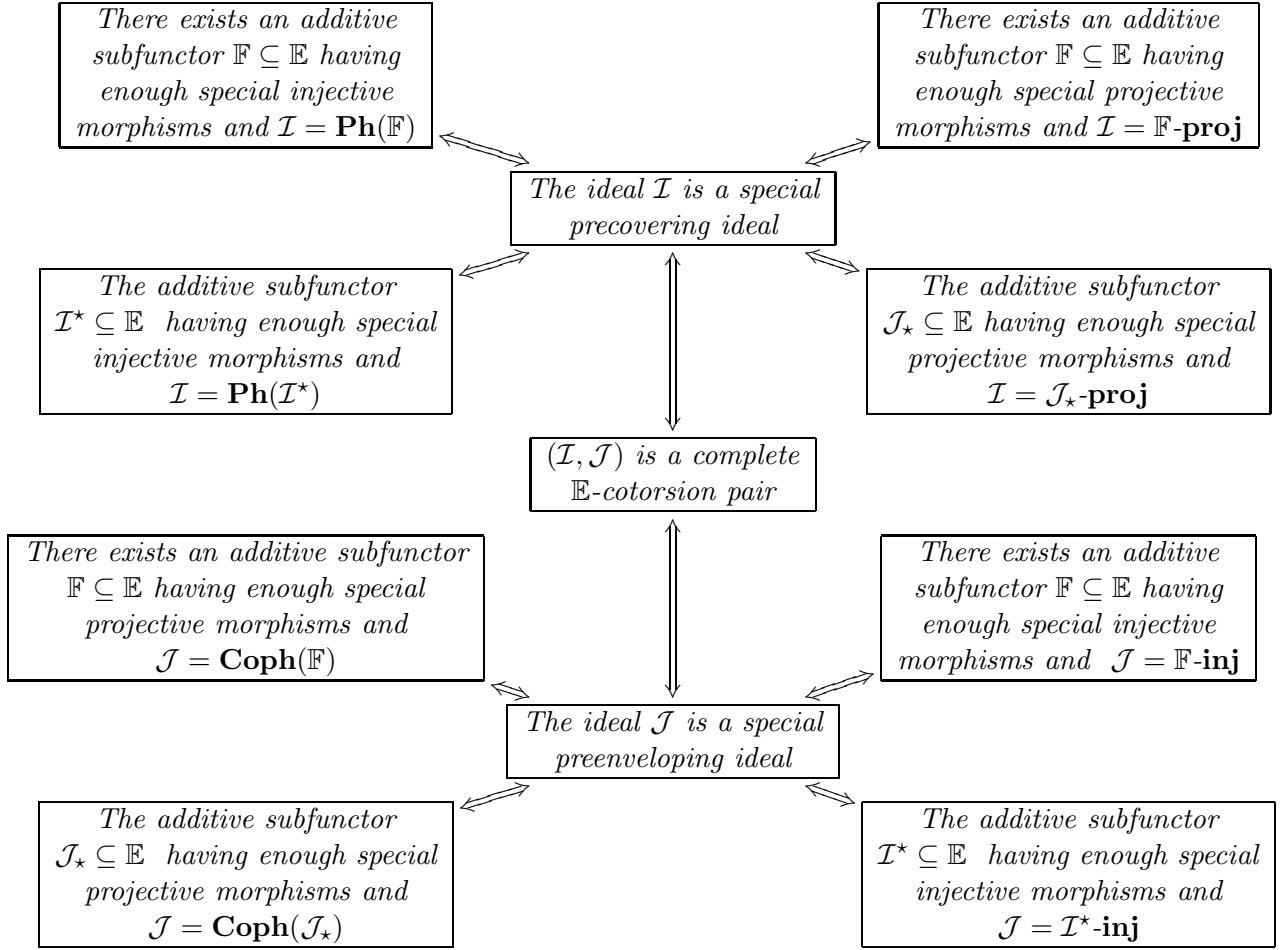
\xymatrix@=0.1cm{
{\begin{tabular}{|c|}
  \hline
\mbox{There exists an additive} \\ \mbox{subfunctor $\F\subseteq\E$ having}\\\mbox{enough special injective}\\{morphisms and $\I=\Ph(\F)$} \\
  \hline
\end{tabular}} \ar@{<=>}[dr]&&\ar@{<=>}[dl] {\begin{tabular}{|c|}
  \hline
\mbox{There exists an additive} \\ \mbox{subfunctor $\F\subseteq\E$ having}\\\mbox{enough special projective}\\\mbox{morphisms and }$\I=\F\mbox{-}\proj$ \\
  \hline
\end{tabular}} \\
 & {\begin{tabular}{|c|}
  \hline
        \mbox{The ideal $\I$ is a special}\\\mbox{precovering ideal} \\
  \hline
\end{tabular}} \ar@{<=>}[dr]\ar@{<=>}[dl]&   \\
  {\begin{tabular}{|c|}
  \hline
\mbox{The additive subfunctor} \\$\I^\star\subseteq\E$ \mbox{ having enough special}\\\mbox{injective morphisms and}\\$\I=\Ph(\I^\star)$ \\
  \hline
\end{tabular}} && {\begin{tabular}{|c|}
  \hline
\mbox{The additive subfunctor} \\ $\J_\star\subseteq\E$\mbox{ having enough special}\\\mbox{projective morphisms and}\\$\I=\J_\star\mbox{-}\proj$ \\
  \hline
\end{tabular}} \\
&  {\begin{tabular}{|c|}
  \hline
      \mbox{$(\I,\J)$ is a complete} \\ \mbox{$\E$-cotorsion pair}\\
  \hline
\end{tabular}}\ar@{<=>}[uu]\ar@{<=>}[dd] &\\
 {\begin{tabular}{|c|}
  \hline
\mbox{There exists an additive subfunctor} \\ \mbox{$\F\subseteq\E$ having enough special}\\\mbox{projective morphisms and} \\$\J=\Coph(\F)$\\
  \hline
\end{tabular}} \ar@{<=>}[dr]&&\ar@{<=>}[dl]{\begin{tabular}{|c|}
  \hline
\mbox{There exists an additive} \\ \mbox{subfunctor $\F\subseteq\E$ having}\\\mbox{enough special injective} \\ \mbox{morphisms and } $\J=\F\mbox{-}\inj$\\
  \hline
\end{tabular}}  \\
 & {\begin{tabular}{|c|}
  \hline
        \mbox{The ideal $\J$ is a special}\\\mbox{preenveloping ideal} \\
  \hline
\end{tabular}} \ar@{<=>}[dr]\ar@{<=>}[dl]&   \\
  {\begin{tabular}{|c|}
  \hline
      \mbox{The additive subfunctor} \\$\J_\star\subseteq\E$ \mbox{ having enough special}\\\mbox{projective morphisms and}\\$\J=\Coph(\J_\star)$ \\
  \hline
\end{tabular}} && {\begin{tabular}{|c|}
  \hline
      \mbox{The additive subfunctor} \\$\I^\star\subseteq\E$ \mbox{ having enough special}\\\mbox{injective morphisms and}\\$\J=\I^\star\mbox{-}\inj$ \\
  \hline
\end{tabular}}
}$$
\end{theorem}

The above theorem shows that if $(\C,\E,\s)$ is an extriangulated category with enough injective objects and projective objects,
then we have the following bijective correspondences.
$$(\bigstar\bigstar) \ \
\xymatrix@=1.5cm{
{\begin{tabular}{|c|}
  \hline
      \mbox{all special precovering ideals} \\ \mbox{of $\C$}\\
  \hline
\end{tabular}}\ar@<+4pt>[r]^{\!\!\!\!\!\!\!\!\!\!\!\!\!\!\!(-)^\star}\ar@<+4pt>[d]^{(-)^{\perp_\E}}&
\ar@<+4pt>[d]^G\ar@<+4pt>[l]^{\!\!\!\!\!\!\!\!\!\!\!\!\!\!\!\!\Ph(-)}\ar[ld]_{\!\!\!\!\!\!\!\!\!\!\!\!\!\!\!\!\!\!\!\!\!\!\!\!\!\!\!\!\!\!\!\!\!\!\!\!\!
\!\!\!\!\!\!\!\!\!\!\!\!\!\!\!\!\!\!\!\!\!\!\!\!\!\!\!\!\!\!\!\!\!\!\!\!\!\!\!\!\!\!\!(-)\mbox{-}\proj}{\begin{tabular}{|c|}
  \hline
        \mbox{all additive subfunctors of $\E$ having}\\\mbox{enough special injective morphisms} \\
  \hline
\end{tabular}}\\
{\begin{tabular}{|c|}
  \hline
      \mbox{all special preenveloping ideals} \\ \mbox{of $\C$}\\
  \hline
\end{tabular}}\ar@<+4pt>[r]^{\!\!\!\!\!\!\!\!\!\!\!\!\!\!\!(-)_\star}\ar@<+4pt>[u]^{^{\perp_\E}(-)}&\ar@<+4pt>[u]^F\ar@<+4pt>[l]^{\!\!\!\!\!\!\!\!\!\!\!\!\!\!\!\!\Coph(-)}
\ar[lu]_{~~~~~~~~~~~~~~~~(-)\mbox{-}\inj}{\begin{tabular}{|c|}
  \hline
        \mbox{all additive subfunctors of $\E$ having}\\\mbox{enough special projective morphisms} \\
  \hline
\end{tabular}}}
$$
Here $F={(-)^\star}\circ{^{\perp_\E}(-)}\circ\Coph(-)$ and $G={(-)_\star}\circ{{(-)^{\perp_\E}}}\circ{\Ph(-)}$.

We end this section with some applications of the obtained results above.

\begin{theorem}\label{4.6}
Let  $(\C,\E,\s)$ be an extriangulated category with enough injective objects and projective morphisms.
If an additive subfunctor $\F\subseteq\E$ has enough injective morphisms, then we have
\begin{enumerate}
\item[(1)] The pair $({^{\perp_\E}(\F\mbox{-}\inj)},{(^{\perp_\E}(\F\mbox{-}\inj))}^{\perp_\E})$ of ideals generated by
$\F\mbox{-}\inj$ is a complete $\E$-cotorsion pair of ideals.
\item[(2)] $\Ph(\F)^{\perp_\E}=\Ph(\F)^\star\mbox{-}\inj$, and $\Ph(\F)^{\perp_\E}$ is the minimal ideal containing
$\F\mbox{-}\inj$ and satisfying the following property (C): Let $\I$ be an ideal and consider a morphism of $\E$-triangles
$$\xymatrix{
A \ar[r]^a\ar[d]^f & B \ar[r]^b\ar[d]^g & E\ar[d]^h\ar@{-->}[r]^{\gamma}& \\
X \ar[r]^x & Y \ar[r]^y & Z \ar@{-->}[r]^{\delta}& .}$$
If $f\in\I$ and $E$ is an injective object, then $g\in\I$.
\item[(3)] The additive subfunctor $\Ph(\F)^\star\subseteq\E$ is the maximal additive subfunctor of $\F$ having enough special injective morphisms.
\end{enumerate}
\end{theorem}

\begin{proof}
(1) By Proposition \ref{3.12}, we have ${^{\perp_\E}(\F\mbox{-}\inj)}=\Ph(\F)$. By Theorem \ref{4.2}, $\Ph(\F)$ is a special precovering ideal.
Moreover, by Theorem \ref{4.4}, $\Ph(\F)^{\perp_\E}={(^{\perp_\E}(\F\mbox{-}\inj))}^{\perp_\E}$ is a special preenveloping ideal.
Thus $({^{\perp_\E}(\F\mbox{-}\inj)},{(^{\perp_\E}(\F\mbox{-}\inj))}^{\perp_\E})$  is a complete $\E$-cotorsion pair of ideals.

(2) By Proposition \ref{3.4}, $\F\mbox{-}\inj\subseteq\Ph(\F)^{\perp_\E}=\Ph(\F)^\star\mbox{-}\inj$. By Lemma \ref{4.3},
$\Ph(\F)^{\perp_\E}$ satisfies the property (C). Now let $\mathcal{J}$ be an ideal of $\C$ containing $\F\mbox{-}\inj$ and
satisfying the  property (C). We will show that $\Ph(\F)^{\perp_\E}\subseteq \mathcal{J}$. To do it, let $j\in\Ph(\F)^{\perp_\E}$ with $j:A\ra J$.
Consider the same commutative diagram as in the proof of Theorem \ref{4.4}. Since $\F\subseteq\E$ has enough injective morphisms,
we can adjust the morphism $g:Y\ra W$ to be in $\F\mbox{-}\inj$. By the property (C), we have that $k\in\mathcal{J}$ and $a=ke'\in\mathcal{J}$.
Moreover, by Theorem \ref{4.4}, the morphism $a:A\ra B$ is a $\Ph(\F)^{\perp_\E}$-preenvelope of $A$, it factors through $j$, that is,
there exists $b:B\ra J$ such that $j=ba$, and thus $j\in\mathcal{J}$, as desired.

(3) Clearly, $\Ph(\F)^\star\subseteq\F$. Now since $\Ph(\F)$ is a special precovering ideal by Theorem \ref{4.2}, $\Ph(\F)^\star$ is an
additive subfunctor having enough special injective morphisms by the correspondence $(\bigstar)$. Suppose that $\F'\subseteq\F$ is an
additive subfunctor having enough special injective morphisms. To show $\F'\subseteq \Ph(\F)^\star$, it suffices to show that every $\F'$-triangle
$\xymatrix@C=0.5cm{A\ar[r] & B \ar[r] & C\ar@{-->}[r]^{\delta}&}$ is a $\Ph(\F)^\star$-triangle.

Let $e:A\ra X$ be a special $\F'$-injective $\F'$-inflation. Then we have a morphism of $\F'$-triangles
$$\xymatrix{
A \ar[r]^a\ar[d]^e & B \ar[r]\ar[d] & C\ar@{=}[d]\ar@{-->}[r]^{\delta}& \\
X \ar[r] &Y \ar[r] & C \ar@{-->}[r]^{e_\star\delta}& .}$$
Since $e$ is $\F'$-injective, we have that $e_\star\delta=0$ and there exists $b:B\ra X$ such that $e=ba$. This also induces
the following commutative diagram
$$\xymatrix{
A \ar[r]^a\ar@{=}[d] & B \ar[r]\ar[d]^b & C\ar@{-->}[r]^{\delta}& \\
A \ar[r]^e &X \ar[r] & Z \ar@{-->}[r]^{\gamma}& .}$$
By (ET3), we get a morphism of $\F'$-triangles
$$\xymatrix{
A \ar[r]^a\ar@{=}[d] & B \ar[r]\ar[d]^b & C\ar[d]^c\ar@{-->}[r]^{\delta}& \\
A \ar[r]^e &X \ar[r] & Z \ar@{-->}[r]^{\gamma}& .}$$
On the other hand, since $e:A\ra X$ is a special $\F'$-injective $\F'$-inflation, by definition there exists a morphism of $\E$-triangles
$$\xymatrix{
A \ar[r]^e\ar@{=}[d] & X \ar[r]\ar[d] & Z\ar[d]^j\ar@{-->}[r]^{\gamma}& \\
A \ar[r] &X' \ar[r] & Z' \ar@{-->}[r]^{\gamma'}& }$$ with $j\in\Ph(\F')\subseteq\Ph(\F)$.
Thus we get a morphism of $\E$-triangles
$$\xymatrix{
A \ar[r]^a\ar@{=}[d] & B \ar[r]\ar[d] & C\ar[d]^{jc}\ar@{-->}[r]^{\delta}& \\
A \ar[r] &X' \ar[r] & Z' \ar@{-->}[r]^{\gamma'}& }$$ with $jc\in\Ph(\F)$.
This shows that $\delta=(jc)^\star\gamma'\in \Ph(\F)^\star$, as desired.
\end{proof}

\begin{corollary}\label{4.7}
Let $(\C,\E,\s)$ be an extriangulated category with enough injective objects and projective morphisms.
If $\F\subseteq\E$ is an additive subfunctor having enough injective morphisms, then the following statement are equivalent.
\begin{enumerate}
\item[(1)] The subfunctor $\F$ has enough special injective morphisms.
\item[(2)] $\F=\Ph(\F)^\star$.
\item[(3)] $\Ph(\F)^{\perp_\E}=\F\mbox{-}\inj$.
\end{enumerate}
\end{corollary}

\begin{proof}
(1) $\Rightarrow$ (2) It follows from Theorem \ref{4.6}(3).

(2) $\Rightarrow$ (3) By Proposition \ref{3.4}(2).

(3) $\Rightarrow$ (1) By Theorem \ref{4.2}, $\Ph(\F)$ is a special precovering ideal. By Theorem \ref{4.5}(II),
$(\Ph(\F),\Ph(\F)^{\perp_\E})=(\Ph(\F),\F\mbox{-}\inj)$ is a complete $\E$-cotorsion pair, and hence $\F\mbox{-}\inj$ is a special
preenveloping ideal, that is, for any $A\in\C$, there exists a morphism of $\E$-triangles
$$\xymatrix{
A \ar[r]^e\ar@{=}[d] & B \ar[r]\ar[d] & C\ar[d]^{j}\ar@{-->}[r]^{\delta}& \\
A \ar[r] &Y \ar[r] & Z \ar@{-->}[r]^{\gamma}& }$$
with $e\in\F\mbox{-}\inj$ and $j\in {^{\perp_\E}(\F\mbox{-}\inj)}$. Moreover, since $(\Ph(\F),\F\mbox{-}\inj)$ is an $\E$-cotorsion pair,
we have that ${^{\perp_\E}(\F\mbox{-}\inj)}=\Ph(\F)$ and $j\in \Ph(\F)$. Thus $\F$ has enough special injective morphisms, as desired.
\end{proof}

\section{The correspondences for object ideals}

Let $\I$ be a class of morphisms in $\C$. We write $\Ob(\I):=\{A\in\C\mid \id_A\in\I\}$, and denote by $<\I>$ the smallest ideal of $\C$
containing $\I$. If $\I=<\Ob(\I)>$, then we call $\I$ an \emph{object ideal}, that is, it is generated by itself objects.
An object $A\in\C$ is called {\it $\F$-injective} if $\id_A\in\F\mbox{-}\inj$. It is easy to check that an object  $A\in\C$ is
$\F$-injective if and only if it is injective with respect to all $\F$-triangles.

Let $(\C,\E,\s)$ be an extriangulated category with enough projective morphisms and $\F\subseteq\E$ an additive subfunctor having
enough injective objects. Then for any $C\in\C$, there is an $\E$-triangle $\xymatrix@C=0.5cm{
K \ar[r] & P \ar[r]^p & C\ar@{-->}[r]^{\gamma}&}$ with $p$ an $\E$-projective morphism. For the object $K$, by assumption there exists
an $\F$-inflation $e:K\ra E$ with $E$ an $\F$-injective object. Then we get a morphism of $\E$-triangles
$$\xymatrix{
K \ar[r]\ar[d]^e & P \ar[r]^p\ar[d] & C\ar@{=}[d]\ar@{-->}[r]^{\gamma}& \\
E \ar[r] & Y \ar[r]^\varphi & C \ar@{-->}[r]^{e_\star\gamma}& .}$$
Since $E$ is an $\F$-injective object, that is, $\id_E\in\F\mbox{-}\inj$, we have $e=\id_E e\in\F\mbox{-}\inj$.
Thus as in the proof of Theorem \ref{4.2}, the morphism $\varphi$ is an $\F$-phantom morphism.
Therefore, for any $C\in\C$, there always exists an $\E$-triangle
$\xymatrix@C=0.5cm{E \ar[r] & Y \ar[r]^\varphi & C \ar@{-->}[r]^{\delta}& }$ with $\varphi$ an $\F$-phantom morphism and $E$
an $\F$-injective object. Moreover, since $\F\mbox{-}\inj\subseteq\Ph(\F)^{\perp_\E}$, the object $E$ is also in $\Ph(\F)^{\perp_\E}$.
This allows us to give the following definition.

\begin{definition}
{\rm Let $\I$ be an ideal of $\C$. We call a morphism $i:X\ra C$ in $\I$ an \emph{object-special $\I$-precover} of $C$ if there exists an
$\E$-triangle $\xymatrix@C=0.5cm{A \ar[r] & X \ar[r]^i & C \ar@{-->}[r]^{\delta}& }$ with $A\in\I^{\perp_\E}$.}
\end{definition}

By a trivial morphism of $\E$-triangle
$$\xymatrix{
A \ar[r]\ar@{=}[d] & X \ar[r]\ar@{=}[d] & C\ar@{=}[d]\ar@{-->}[r]^{\delta}& \\
A \ar[r] & X \ar[r]^i & C \ar@{-->}[r]^{\delta}& ,}$$
we have that any object-special $\I$-precover is a special $\I$-precover. In the following, we give a sufficient condition such that
a special precovering ideal $\I$ is an object-special precovering ideal, that is, any object in $\C$ admits an object-special $\I$-precover.

\begin{proposition}\label{5.2}
Let $\I$ be a special precovering ideal. If $\I^{\perp_\E}$ is an object ideal, then $\I$ is an object-special precovering ideal.
\end{proposition}

\begin{proof}
Let $C\in\C$, and take a special $\I$-precover $i':X'\ra C$ which comes from a morphism of $\E$-triangles
$$\xymatrix{
A \ar[r]\ar[d]^j & B \ar[r]\ar[d] & C\ar@{=}[d]\ar@{-->}[r]^{\delta}& \\
A' \ar[r] & X' \ar[r]^{i'} & C \ar@{-->}[r]^{j_\star\delta}& }$$ with $j\in\I^{\perp_\E}$.
Since $\I^{\perp_\E}$ is an object ideal by assumption, there exist $Y\in\I^{\perp_\E}$ and morphisms $j_1:A\ra Y$, $j_2:Y\ra A'$
such that $j=j_2 j_1$. Then by the equality $j_\star \delta={j_2}_\star{j_1}_\star\delta$, we can decompose the above morphism of $\E$-triangles
to the following morphisms of $\E$-triangles
$$\xymatrix{
A \ar[r]\ar[d]^{j_1} & B \ar[r]\ar[d] & C\ar@{=}[d]\ar@{-->}[r]^{\delta}& \\
Y \ar[r]\ar[d]^{j_2} & X \ar[r]^i\ar[d]^k & C\ar@{=}[d]\ar@{-->}[r]^{{j_1}_\star\delta}& \\
A' \ar[r] & X' \ar[r]^{i'} & C \ar@{-->}[r]^{{j}_\star\delta}&}$$
with $i=i'k\in\I$. Thus $i$ is an object-special $\I$-precover of $C$.
\end{proof}

In view of Proposition \ref{5.2}, it is natural to ask when the right perpendicularity of
a special precovering ideal is an object ideal. To study it, we consider the following

\begin{enumerate}
\item[(J)] Let $\I$ be an ideal of $\C$. There exists an object ideal $\J\subseteq\I^{\perp_\E}$
such that any $C\in \C$ admits an $\I$-precover $i:X\ra C$ together with an $\E$-triangle
$\xymatrix@C=0.5cm{
A \ar[r] & X \ar[r]^i & C\ar@{-->}[r]^{\delta}& }$, where $A\in\J$.
\end{enumerate}

Let $\mathcal{X}$ and $\mathcal{Y}$ be two classes of objects in $\C$. We write
$$\mathcal{X}\diamond\mathcal{Y}:=\{Z\in\C\mid \mbox{ there exists an $\E$-triangle }
\xymatrix@C=0.5cm{
X \ar[r] & Z \ar[r] & Y\ar@{-->}[r]^{\delta}& }\mbox{ with }X\in\mathcal{X}\mbox{ and } Y\in\mathcal{Y}\}.$$

\begin{proposition}\label{5.3}
Let $(\C,\E,\s)$ be an extriangulated category with enough injective objects and $\I$ be a special precovering ideal of $\C$.
The condition (J) is satisfied if and only if $\I^{\perp_\E}$ is an object ideal; in this case, we have
$$\I^{\perp_\E}=<\Ob(\J)\diamond\Ob(\E\mbox{-}\inj)>.$$
\end{proposition}

\begin{proof}
The sufficiency is trivial. In the following, we prove the necessity.

Let $Z\in\Ob(\J)\diamond\Ob(\E\mbox{-}\inj)$, that is, there exists an $\E$-triangle
$\xymatrix@C=0.5cm{
X \ar[r] & Z \ar[r] & Y\ar@{-->}[r]^{\delta}& }$ with $X\in\J$ and $Y$ an $\E$-injective object. By assumption,
we have $X\in\I^{\perp_\E}$, and hence $Z\in\I^{\perp_\E}$ by Lemma \ref{4.3}. This shows that
$<\Ob(\J)\diamond\Ob(\E\mbox{-}\inj)>\subseteq\I^{\perp_\E}$.

Conversely, let $A\in\C$. Then there exists an $\E$-triangle $\xymatrix@C=0.5cm{
A \ar[r] & E \ar[r]^e & C\ar@{-->}[r]^{\delta}& }$ with $E$ an injective object by assumption.
For the object $C$, by (J) there exists an $\E$-triangle $\xymatrix@C=0.5cm{
K \ar[r] & X \ar[r]^i & C\ar@{-->}[r]^{\gamma}& }$ with $K\in\J$ and $i:X\ra C$ an $\I$-precover of $C$. By (ET4),
we get the following commutative diagram
$$\xymatrix{&K\ar@{=}[r]\ar[d]&K\ar[d]&\\
A\ar[r]^a\ar@{=}[d]&Z\ar[r]\ar[d]&X\ar[d]^i\ar@{-->}[r]^{i^\star\delta}&\\
A \ar[r] & E \ar[r]^e \ar@{-->}[d]^{e^\star\gamma}& C\ar@{-->}[r]^{\delta}\ar@{-->}[d]^{\gamma}&\\
&&.&}$$
By the middle column in the above diagram, we have that $Z\in\Ob(\J)\diamond\Ob(\E\mbox{-}\inj)$. On the other hand,
as in the proof of Theorem \ref{4.4}, the morphism $a:A\ra Z$ is a special $\I^{\perp_\E}$-preenvelope of $A$;
in particular, it is an $\I^{\perp_\E}$-preenvelope of $A$. Thus any $f\in\I^{\perp_\E}(A,B)$ factors through $a$,
that is, there exists $b:Z\ra B$ such that $f=ba$, which shows that $f\in<\Ob(\J)\diamond\Ob(\E\mbox{-}\inj)>$.
Thus $\I^{\perp_\E}\subseteq<\Ob(\J)\diamond\Ob(\E\mbox{-}\inj)>$, and therefore
$\I^{\perp_\E}=<\Ob(\J)\diamond\Ob(\E\mbox{-}\inj)>$; in particular, $\I^{\perp_\E}$ is an object ideal.
\end{proof}

By Propositions \ref{5.2} and \ref{5.3}, we immediately have the following

\begin{corollary}
Let $(\C,\E,\s)$ be an extriangulated category with enough injective objects. If the property (J) is satisfied,
then any special precovering ideal of $\C$ is an object-special precovering ideal.
\end{corollary}

The additive subfunctor $\F$ of $\E$ is said to \emph{have enough special injective objects} if for any $A\in\C$,
there exists an $\F$-triangle $\xymatrix@C=0.5cm{
A \ar[r]^e & B \ar[r] & C\ar@{-->}[r]^{\delta}&}$ with $B\in\F\mbox{-}\inj$, together with a morphism of $\E$-triangles
$$\xymatrix{
A \ar[r]^e\ar@{=}[d] & B \ar[r]\ar[d] & C\ar[d]^{\varphi}\ar@{-->}[r]^{\delta}& \\
A \ar[r] & B' \ar[r]^i & C' \ar@{-->}[r]^{\delta'}& ,}$$
where $\varphi$ is an $\F$-phantom morphism.

\begin{theorem}\label{5.5}
Let  $(\C,\E,\s)$ be an extriangulated category. Then we have the following implications.
$$
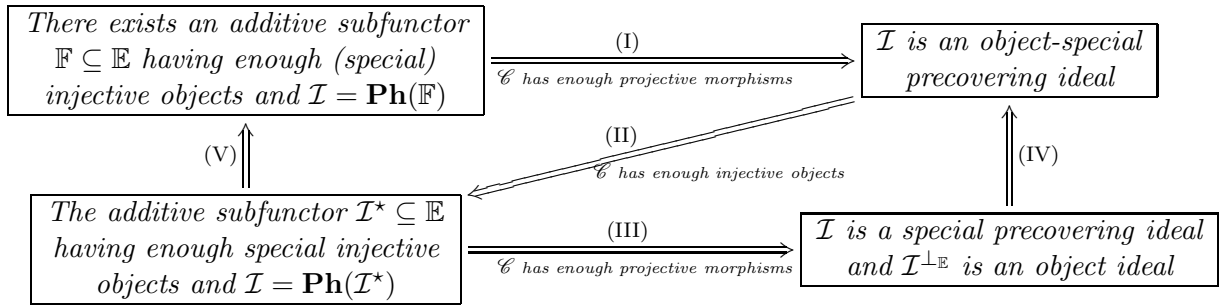
\xymatrix{
{\begin{tabular}{|c|}
  \hline
      \mbox{There exists an additive subfunctor} \\ \mbox{$\F\subseteq\E$ having enough (special)}\\\mbox{injective objects and $\I=\Ph(\F)$} \\
  \hline
\end{tabular}}\ar@{=>}[rrrr]^{\rm(I)}_{\tiny ~~~~~ \C\mbox{ has enough projective morphisms}}&&&&{\begin{tabular}{|c|}
  \hline
        \mbox{$\I$ is an object-special}\\\mbox{precovering ideal} \\
  \hline
\end{tabular}}\ar@{=>}[lllld]_{\rm(II)}^{~~~~~~~~~\tiny\C \mbox{ has enough injective objects}}\\
{\begin{tabular}{|c|}
  \hline
      \mbox{The additive subfunctor $\I^\star\subseteq\E$} \\ \mbox{having enough special injective}\\\mbox{objects and $\I=\Ph(\I^\star)$} \\
  \hline
\end{tabular}}\ar@{=>}[u]^{\rm(V)}\ar@{=>}[rrrr]^{\rm(III)}_{\tiny ~~~~~ \C\mbox{ has enough projective morphisms}}&&&&{\begin{tabular}{|c|}
  \hline
      \mbox{$\I$ is a special precovering ideal} \\ \mbox{and $\I^{\perp_\E}$ is an object ideal}\\
  \hline
\end{tabular}}\ar@{=>}[u]_{\rm(IV)}}$$
\end{theorem}

\begin{proof}
(I) For any $C\in\C$, there exists an $\E$-triangle
$\xymatrix@C=0.5cm{E \ar[r] & Y \ar[r]^\varphi & C \ar@{-->}[r]^{\delta}& }$ with $\varphi$ an $\F$-phantom morphism and $E$
an $\F$-injective object. Because $\I=\Ph(\F)$ and $\F\mbox{-}\inj\subseteq\Ph(\F)^{\perp_\E}$, we have that $\varphi$ is an
object-special $\I$-precover of $A$. Thus $\I$ is an object-special precovering ideal.

(II) Since  $\I$ is an object-special precovering ideal, it is clearly an special precovering ideal, and hence
$\I=\Ph(\I^{\star})$ by Corollary \ref{3.9}. Let $A\in\C$, by assumption there exists an $\E$-triangle
$\xymatrix@C=0.5cm{A \ar[r] & E \ar[r]^e & C \ar@{-->}[r]^{\delta}& }$ with $E$ an injective object. For the object $C$, since
$\I$ is an object-special precovering ideal, there exists an $\E$-triangle
$\xymatrix@C=0.5cm{K \ar[r] & X \ar[r]^i & C \ar@{-->}[r]^{\gamma}& }$ with $i\in\I$ and $K\in \I^{\perp_\E}$.
By (ET4), we get the following commutative diagram
$$\xymatrix{&K\ar@{=}[r]\ar[d]&K\ar[d]&\\
A\ar[r]^a\ar[d]&Z\ar[r]\ar[d]&X\ar[d]^i\ar@{-->}[r]^{i^\star\delta}&\\
A \ar[r] & E \ar[r]^e \ar@{-->}[d]^{e^\star\gamma}& C\ar@{-->}[r]^{\delta}\ar@{-->}[d]^{\gamma}&\\
&&.&}$$
The middle row in the above diagram is an $\I^\star$-triangle. Moreover, since $K\in \I^{\perp_\E}$ and $E$ is an injective object,
we have $Z\in\I^{\perp_\E}$ by Lemma \ref{4.3}. By Proposition \ref{3.4}, we have $\I^{\perp_\E}=\I^\star\mbox{-}\inj$. Thus
$Z\in\I^\star\mbox{-}\inj$. Since $i\in\I=\Ph(\I^{\star})$, the above diagram shows that $\I^\star$ has enough special injective objects.

(III) Assume that $\I^\star$ has enough special injective objects. Of course, $\I^\star$ has enough injective morphisms, and then by
Theorem \ref{4.2}, $\I=\Ph(\I^\star)$ is a special precovering ideal. By assumption, for any $A\in\C$, there exists an $\E$-triangle
$\xymatrix@C=0.5cm{A \ar[r]^e & E \ar[r] & X \ar@{-->}[r]^{\delta}& }$ with $E$ an $\I^\star$-injective object, together with a morphism
of $\E$-triangles
$$\xymatrix{
A \ar[r]^e\ar@{=}[d] & E \ar[r]\ar[d] & X\ar[d]^{\varphi}\ar@{-->}[r]^{\delta}& \\
A \ar[r] & B \ar[r]^i & C \ar@{-->}[r]^{\delta'}& }
$$
with $\varphi\in\Ph(\I^\star)$. By Proposition \ref{3.12}, we have $\Ph(\I^\star)={^{\perp_\E}(\I^\star\mbox{-}\inj)}$, and hence
$\varphi\in{^{\perp_\E}(\I^\star\mbox{-}\inj)}$, which shows that $e$ is a special $\I^\star$-injective preenvelope of $A$.

By Proposition \ref{3.4}(2), $\I^\star\mbox{-}\inj=\I^{\perp_\E}$. So for any $a:A\ra A'\in \I^{\perp_\E}$, there exists $e':E\ra A'$
such that $a=e'e$. This means that each morphism in $\I^{\perp_\E}$ factors through an $\I^\star$-injective object, and therefore
$\I^{\perp_\E}$ is an object ideal.

(IV) By Proposition \ref{5.2}.

(V) It is trivial.
\end{proof}

The above theorem shows that if $(\C,\E,\s)$ is an extriangulated category with enough injective objects and projective morphisms,
then we have the following bijective correspondence.
$$(\bigstar\bigstar\bigstar) \ \
\xymatrix@C=1.5cm{
{\begin{tabular}{|c|}
  \hline
      \mbox{all object-special} \\ \mbox{precovering ideals of $\C$}\\
  \hline
\end{tabular}}\ar@<+4pt>[r]^{\!\!\!\!\!\!\!\!\!\!\!\!\!\!\!\!\!\!\!\!\!
(-)^\star}&\ar@<+4pt>[l]^{\!\!\!\!\!\!\!\!\!\!\!\!\!\!\!\!\!\!\Ph(-)}{\begin{tabular}{|c|}
  \hline
        \mbox{all additive subfunctors of $\E$ having}\\\mbox{enough special injective objects} \\
  \hline
\end{tabular}}}$$

Note that $(\bigstar\bigstar)$ follows from $(\bigstar)$ and the morphism version of the Salce's lemma.
Now, in view of $(\bigstar\bigstar\bigstar)$, it is natural to pose the following

\begin{question}
Does the Salce's lemma hold for object-special precovering ideals and object-special preenveloping ideals?
\end{question}

\vspace{0.5cm}

{\bf Acknowledgements.}
The first author was  partially supported by the University Postgraduate Research and Innovation Project
of Jiangsu Province 2016 (No. KYZZ16\_0034) and Nanjing University Innovation and Creative Program for PhD candidate (No. 2016011).
The second author was partially supported by NSFC (Grant No. 11571164) and
a Project Funded by the Priority Academic Program Development of Jiangsu Higher Education Institutions.

\end{document}